%% file: input.tex
\documentclass[12pt]{amsart}

\usepackage{amsmath}
\usepackage{amssymb}
\usepackage{ascmac}
\usepackage{amsthm}

\makeatletter

\@addtoreset{equation}{section}
\makeatother 

\theoremstyle{plain}
\newtheorem{thm}{Theorem}[section]
\newtheorem*{thm*}{Theorem}
\newtheorem{prop}[thm]{Proposition}
\newtheorem{lem}[thm]{Lemma}
\newtheorem{cor}[thm]{Corollary}

\theoremstyle{definition}

\theoremstyle{remark}
\newtheorem{rem}{Remark}[section]

\newcommand{\vol}{\operatorname{vol}}
\newcommand{\ric}{\operatorname{Ric}}
\newcommand{\Div}{\operatorname{div}}
\newcommand{\Hess}{\operatorname{Hess}}

\newcommand{\tr}{\operatorname{trace}}

\newcommand{\cut}{\mathrm{Cut}\,}
\newcommand{\inte}{\mathrm{Int}\,}

\newcommand{\bm}{\partial M}

\newcommand{\dm}{D(M,\partial M)}

\newcommand{\expp}{\exp^{\perp}}
\newcommand{\tbp}{T^{\perp}\bm}

\usepackage{latexsym}

\title[rigidity phenomena in manifolds with boundary]{Rigidity phenomena in manifolds\\ with boundary under a lower\\ weighted Ricci curvature bound}
\author{Yohei Sakurai}
\date{April 6, 2017.}
\address{Faculty of Pure and Applied Sciences, University of Tsukuba, Tennodai 1-1-1, Tsukuba, Ibaraki, 305-8571, Japan}
\email{sakurai@math.tsukuba.ac.jp}
\thanks{Research Fellow of Japan Society for the Promotion of Science for 2014-2016}
\subjclass[2010]{Primary 53C20}
\keywords{Manifold with boundary; Weighted Ricci curvature}
\begin{document}
\maketitle

\begin{abstract}
We study Riemannian manifolds with boundary under a lower $N$-weighted Ricci curvature bound for $N$ at most $1$,
and under a lower weighted mean curvature bound for the boundary.
We examine rigidity phenomena in such manifolds with boundary.
We conclude a volume growth rigidity theorem for the metric neighborhoods of the boundaries,
and various splitting theorems.
We also obtain rigidity theorems for the smallest Dirichlet eigenvalues for the weighted $p$-Laplacians.
\end{abstract}

\input{section1}
\input{section2}

\input{section3}
\input{section4}

\input{section5}

\input{section6}

\end{document}

%% file: section1.tex
\section{Introduction}\label{sec:Introduction}
In this paper,
we study Riemannian manifolds with boundary
under a lower weighted Ricci curvature bound,
and under a lower weighted mean curvature bound for the boundary.
We develop the preceding studies of the author \cite{Sa2}.
As explained below,
we examine rigidity phenomena in such manifolds with boundary
beyond the usual weighted setting.

For $n\geq 2$,
let $M$ be an $n$-dimensional Riemannian manifold with or without boundary with Riemannian metric $g$,
and let $f:M\to \mathbb{R}$ be a smooth function.
We denote by $\ric_{g}$ the Ricci curvature defined by $g$,
by $\nabla f$ the gradient of $f$,
and by $\Hess f$ the Hessian of $f$.
For $N\in (-\infty,\infty]$,
the \textit{$N$-weighted Ricci curvature} $\ric^{N}_{f}$ is defined as
\begin{equation}\label{eq:def of weighted Ricci curvature}
\ric^{N}_{f}:=\ric_{g}+\Hess f-\frac{\nabla f \otimes \nabla f}{N-n}
\end{equation}
if $N \in (-\infty,\infty)\setminus \{n\}$;
otherwise,
if $N=\infty$,
then $\ric^{N}_{f}:=\ric_{g}+\Hess f$;
if $N=n$,
and if $f$ is a constant function,
then $\ric^{N}_{f}:=\ric_{g}$;
if $N=n$,
and if $f$ is not constant,
then $\ric^{N}_{f}:=-\infty$ (\cite{BE}).
We notice that
the parameter $N$ has been usually chosen from $[n,\infty]$.

On manifolds without boundary under a lower $N$-weighted Ricci curvature bound,
many results have been already known in the usual weighted case of $N \in [n,\infty]$ (see e.g., \cite{Lo}, \cite{LV1}, \cite{LV2}, \cite{Q}, \cite{St2}, \cite{St3}, \cite{WW}).
Recently,
in the complemental weighted case of $N \in (-\infty,n)$,
several geometric properties have begun to be studied (see \cite{K}, \cite{KM}, \cite{M}, \cite{O3}, \cite{OT1}, \cite{OT2}, \cite{WY}).
Wylie \cite{W} has obtained a splitting theorem of Cheeger-Gromoll type (cf. \cite{CG}) in the complementary weighted case of $N \in (-\infty,1]$,
and asked a question whether
the splitting theorem can be extended to the remaining case of $N \in (1,n)$.

For manifolds with boundary under a lower $N$-weighted Ricci curvature bound,
and under a lower weighted mean curvature bound for the boundary,
the author \cite{Sa2} has studied rigidity phenomena in the usual weighted case of $N \in [n,\infty]$.
In the present paper,
we produce rigidity theorems in the complementary weighted case of $N \in (-\infty,1]$.
Our rigidity theorems in the case of $N \in (-\infty,1]$ give natural extensions of the corresponding results in \cite{Sa2}.

To prove our rigidity theorems,
we develop comparison theorems.
We prove Laplacian comparison theorems for the distance function from the boundary,
and volume comparison theorems for metric neighborhoods of the boundary.
The author \cite{Sa2} has shown such comparison theorems in the usual weighted case of $N \in [n,\infty]$.
For manifolds with boundary of non-negative $N$-weighted Ricci curvature,
and of non-negative weighted mean curvature for the boundary,
Wylie \cite{W} has shown a Laplacian comparison inequality for the distance function from a connected component of the boundary in the weighted case of $N \in (-\infty,1]$.
To conclude our comparison theorems,
we need slightly more complicated calculations than that done by the author \cite{Sa2},
and by Wylie \cite{W}.
Under an assumption concerning a subharmonicity of the distance function from the boundary,
we derive our rigidity theorems
from studies of the equality cases in our comparison theorems.

\subsection{Setting}\label{sec:setting}
We summarize our setting as follows:
For $n\geq 2$,
let $M$ be an $n$-dimensional, 
connected complete Riemannian manifold with boundary with Riemannian metric $g$.
The boundary $\bm$ is assumed to be smooth.
We denote by $d_{M}$ the Riemannian distance on $M$ induced from the length structure determined by $g$.
Let $f:M\to \mathbb{R}$ be a smooth function.
For the Riemannian volume measure $\vol_{g}$ on $M$,
let
\begin{equation}\label{eq:weighted measure}
m_{f}:=e^{-f}\, \vol_{g}.
\end{equation}
For $N\in (-\infty,\infty]$,
we denote by $\ric^{N}_{f}$ the $N$-weighted Ricci curvature (see (\ref{eq:def of weighted Ricci curvature})).
We note that
for $N_{1},N_{2} \in (-\infty,\infty]\setminus \{n\}$ with $N_{1} \leq N_{2}$,
if $N_{1},N_{2} \in (n,\infty]$ or $N_{1},N_{2} \in (-\infty,n)$,
then $\ric^{N_{1}}_{f}\leq \ric^{N_{2}}_{f}$;
if $N_{1} \in (-\infty,n)$ and $N_{2} \in (n,\infty]$,
then $\ric^{N_{2}}_{f}\leq \ric^{N_{1}}_{f}$.
We denote by $\ric^{N}_{f,M}$ the infimum of $\ric^{N}_{f}$ on the unit tangent bundle on the interior $\inte M$ of $M$.
For $x\in \bm$,
we denote by $u_{x}$ the unit inner normal vector on $\bm$ at $x$.
Let $H_{x}$ denote the mean curvature of $\bm$ at $x$ defined as the trace of the shape operator of $u_{x}$.
The \textit{$f$-mean curvature} $H_{f,x}$ at $x$ is defined by
\begin{equation}\label{eq:weighted mean curvature}
H_{f,x}:=H_{x}+g\left((\nabla f)_{x},u_{x}\right).
\end{equation}
We put $H_{f,\bm}:=\inf_{x\in \bm} H_{f,x}$.
Our main subject is a weighted Riemannian manifold $(M,d_{M},m_{f})$ with boundary such that
for $\kappa,\lambda \in \mathbb{R}$
and for $N \in (-\infty,1]$
we have $\ric^{N}_{f,M}\geq \kappa$ and $H_{f,\bm}\geq \lambda$.

\subsection{Volume growth rigidity}\label{sec:Volume growth rigidity}
Let $\rho_{\bm}:M\to \mathbb{R}$ be the distance function from $\bm$ defined as $\rho_{\bm}(p):=d_{M}(p,\bm)$.
For $r\in(0,\infty)$,
we put $B_{r}(\bm):=\{\,p\in M \mid \rho_{\bm}(p) \leq r\,\}$.
For $x\in \bm$,
let $\gamma_{x}:[0,T)\to M$ be the geodesic with initial conditions $\gamma_{x}(0)=x$ and $\gamma_{x}'(0)=u_{x}$.
We define a function $\tau:\bm \to \mathbb{R} \cup \{\infty\}$ by
\begin{equation}\label{eq:cut point}
\tau(x):=\sup \{t \in(0,\infty) \mid \rho_{\bm}(\gamma_{x}(t))=t\}.
\end{equation}
We define a function $F_{x}:[0,\tau(x)]\setminus \{\infty\} \to (0,\infty)$ by 
\begin{equation}\label{eq:warping function}
F_{x}(t):=e^{\frac{f(\gamma_{x}(t))-f(x)}{n-1}}.
\end{equation}
Notice that
if $f$ is constant,
then $F_{x}$ is equal to $1$.
For $\kappa,\lambda \in \mathbb{R}$,
we say that
$\kappa$ and $\lambda$ satisfy the \textit{subharmonic-condition} if
\begin{equation*}
\inf_{x \in \bm}\,\inf_{t \in (0,\tau(x))}\, \kappa \int^{t}_{0}\,F^{2}_{x}(s)\,ds \geq -\lambda.
\end{equation*}
We remark that
if $\kappa$ and $\lambda$ satisfy the subharmonic-condition,
then subharmonicity of $\rho_{\bm}$ is derived from $\ric^{N}_{f,M} \geq \kappa$ and $H_{f,\bm} \geq \lambda$
in the case of $N \in (-\infty,1]$ (see Lemma \ref{lem:Laplacian comparison}).
Note that
if $\kappa,\lambda \in [0,\infty)$,
then they satisfy the subharmonic-condition.
We denote by $h$ the induced Riemnnian metric on $\bm$.
For the Riemannian volume measure $\vol_{h}$ on $\bm$ induced from $h$,
we put $m_{f,\bm}:=e^{-f|_{\bm}}\, \vol_{h}$.

For an interval $I$,
and for a Riemannian manifold $M_{0}$ with Riemannian metric $g_{0}$,
let $\Phi:I \times M_{0} \to \mathbb{R}$ be a positive smooth function.
For each $x \in M_{0}$,
let $\Phi_{x}:I \to \mathbb{R}$ be the function defined as $\Phi_{x}(t):=\Phi(t,x)$.
We say that
a Riemannian manifold $(I \times M_{0},dt^{2}+\Phi^{2}_{x}(t)\,g_{0})$ is a \textit{twisted product space}.
When $\tau$ is infinity on $\bm$,
we define $[0,\infty)\times_{F}\bm$ as the twisted product space $([0,\infty)\times \bm,dt^{2}+F^{2}_{x}(t)\,h)$.

For the metric neighborhoods of the boundaries,
we prove an absolute volume comparison theorem of Heintze-Karcher type,
and a relative volume comparison theorem of Bishop-Gromov type (see Subsections \ref{sec:Absolute volume comparison} and \ref{sec:Relative volume comparison}).
We obtain rigidity results concerning the equality cases in those comparison theorems (see Subsection \ref{sec:Volume growth rigidity}).

We conclude the following volume growth rigidity theorem:
\begin{thm}\label{thm:volume growth rigidity}
Let $M$ be a connected complete Riemannian manifold with boundary,
and let $f:M\to \mathbb{R}$ be a smooth function.
Suppose that
$\bm$ is compact.
Let $\kappa \in \mathbb{R}$ and $\lambda \in \mathbb{R}$ satisfy the subharmonic-condition.
For $N \in (-\infty,1]$
we suppose $\ric^{N}_{f,M} \geq \kappa$,
and $H_{f,\bm} \geq \lambda$.
If we have
\begin{equation}\label{eq:volume growth rigidity}
\liminf_{r \to \infty}\frac{m_{f}(B_{r}(\bm))}{r} \geq m_{f,\bm}(\bm),
\end{equation}
then $(M,d_{M})$ is isometric to $([0,\infty) \times_{F} \bm,d_{[0,\infty) \times_{F} \bm})$.
Moreover,
if $N\in (-\infty,1)$,
then for every $x\in \bm$
the function $f\circ \gamma_{x}$ is constant on $[0,\infty)$;
in particular,
$(M,d_{M})$ is isometric to $([0,\infty) \times \bm,d_{[0,\infty) \times \bm})$.
\end{thm}
When $\kappa=0$ and $\lambda=0$,
Theorem \ref{thm:volume growth rigidity} has been proved in the unweighted case in \cite{Sa1},
and in the usual weighted case in \cite{Sa2}.
\begin{rem}
We do not know whether
Theorem \ref{thm:volume growth rigidity} can be extended to the weighted case of $N \in (1,n)$.
\end{rem}
\begin{rem}
Under the same setting as in Theorem \ref{thm:volume growth rigidity},
we always have the following inequality (see Lemma \ref{lem:absolute volume comparison}):
\begin{equation}\label{eq:volume growth upper estimate}
\limsup_{r \to \infty}\frac{m_{f}(B_{r}(\bm))}{r} \leq m_{f,\bm}(\bm).
\end{equation}
Theorem \ref{thm:volume growth rigidity} is concerned with rigidity phenomena.
\end{rem}
We have the following corollary of Theorem \ref{thm:volume growth rigidity}:
\begin{cor}\label{cor:warped volume growth rigidity}
Under the same setting as in Theorem \ref{thm:volume growth rigidity},
if $N=1$ and $\kappa=0$,
and if we have $(\ref{eq:volume growth rigidity})$,
then there exist a function $f_{0}:[0,\infty)\to \mathbb{R}$,
and a Riemannian metric $h_{0}$ on $\bm$
such that $M$ is isometric to a warped product space $([0,\infty)\times \bm,dt^{2}+e^{2\frac{f_{0}(t)}{n-1}}h_{0})$.
\end{cor}

\subsection{Splitting theorems}\label{sec:Splitting theorems}
In our setting,
we show Laplacian comparison theorems for $\rho_{\bm}$,
and study the equality cases (see Section \ref{sec:Comparison theorems}).

By using a Laplacian comparison theorem for $\rho_{\bm}$,
and that for Busemann functions,
we prove the following splitting theorem:
\begin{thm}\label{thm:splitting theorem}
Let $M$ be a connected complete Riemannian manifold with boundary,
and let $f:M\to \mathbb{R}$ be a smooth function such that $\sup f(M)<\infty$.
For $N \in (-\infty,1]$
we suppose $\ric^{N}_{f,M} \geq 0$,
and $H_{f,\bm} \geq 0$.
If for some $x_{0}\in \bm$
we have $\tau(x_{0})=\infty$,
then $(M,d_{M})$ is isometric to $([0,\infty) \times_{F} \bm,d_{[0,\infty) \times_{F} \bm})$.
Moreover,
if $N\in (-\infty,1)$,
then for every $x\in \bm$
the function $f \circ \gamma_{x}$ is constant on $[0,\infty)$;
in particular,
$(M,d_{M})$ is isometric to $([0,\infty) \times \bm,d_{[0,\infty) \times \bm})$.
\end{thm}
In the unweighted case,
Kasue \cite{K2} has proved Theorem \ref{thm:splitting theorem}
under the compactness assumption for the boundary (see also the work of Croke and Kleiner \cite{CK}). 
Theorem \ref{thm:splitting theorem} itself has been proved in the unweighted case in \cite{Sa1},
and in the usual weighted case in \cite{Sa2}.
\begin{rem}
We do not know whether
Theorem \ref{thm:splitting theorem} can be extended to the weighted case of $N \in (1,n)$.
\end{rem}
As a corollary of Theorem \ref{thm:splitting theorem},
we see the following:
\begin{cor}\label{cor:warped splitting theorem}
Under the same setting as in Theorem \ref{thm:splitting theorem},
if $N=1$,
and if for some $x_{0}\in \bm$
we have $\tau(x_{0})=\infty$,
then there exist a function $f_{0}:[0,\infty)\to \mathbb{R}$,
and a Riemannian metric $h_{0}$ on $\bm$ such that
$M$ is isometric to $([0,\infty)\times \bm,dt^{2}+e^{2\frac{f_{0}(t)}{n-1}}h_{0})$.
\end{cor}
In Theorem \ref{thm:splitting theorem},
by applying the Wylie splitting theorem in \cite{W} to the boundary,
we obtain a multi-splitting theorem (see Subsection \ref{sec:Multi-splitting}).
We also generalize a splitting theorem studied in \cite{K2} (and \cite{CK}, \cite{I})
for the case where
boundaries are disconnected (see Subsection \ref{sec:Variants of splitting theorems}).

\subsection{Eigenvalue rigidity}\label{subsec:Eigenvalue rigidity}
For $p\in [1,\infty)$,
the \textit{$(1,p)$-Sobolev space} $W^{1,p}_{0}(M,m_{f})$ \textit{on} $(M,m_{f})$ \textit{with compact support}
is defined as the completion of the set of all smooth functions on $M$ whose support is compact and contained in $\inte M$
with respect to the standard $(1,p)$-Sobolev norm.
We denote by $\Vert \cdot \Vert$ the standard norm induced from $g$,
and by $\Div$ the divergence with respect to $g$.
For $p\in [1,\infty)$,
the \textit{$(f,p)$-Laplacian} $\Delta_{f,p}\, \varphi$ for $\varphi \in W^{1,p}_{0}(M,m_{f})$ is defined by
\begin{equation*}
\Delta_{f,p}\,\varphi:=-e^{f}\,\Div \,\left(e^{-f} \Vert \nabla \varphi \Vert^{p-2}\, \nabla \varphi \right)
\end{equation*}
as a distribution on $W^{1,p}_{0}(M,m_{f})$.
A real number $\mu$ is said to be an \textit{$(f,p)$-Dirichlet eigenvalue} for $\Delta_{f,p}$ on $M$
if there exists $\varphi \in W^{1,p}_{0}(M,m_{f})\setminus \{0\}$
such that $\Delta_{f,p} \varphi=\mu \vert \varphi \vert^{p-2}\,\varphi$ holds  on $\inte M$ in a distribution sense on $W^{1,p}_{0}(M,m_{f})$. 
For $p\in [1,\infty)$,
the \textit{Rayleigh quotient} $R_{f,p}(\varphi)$ for $\varphi \in W^{1,p}_{0}(M,m_{f})\setminus \{0\}$ is defined as
\begin{equation*}
R_{f,p}(\varphi):=\frac{\int_{M}\, \Vert \nabla \varphi \Vert^{p}\,d\,m_{f}}{\int_{M}\,  \vert \varphi \vert^{p}\,d\,m_{f}}.
\end{equation*}
We put $\mu_{f,1,p}(M):=\inf_{\varphi} R_{f,p}(\varphi)$,
where the infimum is taken over all non-zero functions in $W^{1,p}_{0}(M,m_{f})$.
The value $\mu_{f,1,2}(M)$ is equal to the infimum of the spectrum of $\Delta_{f,2}$ on $(M,m_{f})$.
If $M$ is compact,
and if $p\in (1,\infty)$,
then $\mu_{f,1,p}(M)$ is equal to the infimum of the set of all $(f,p)$-Dirichlet eigenvalues on $M$.

Let $p\in (1,\infty)$.
For $D\in (0,\infty)$,
let $\mu_{p,D}$ be the positive minimum real number $\mu$ such that
there exists a function $\varphi:[0,D]\to \mathbb{R}$ satisfying
\begin{equation}\label{eq:model space eigenvalue problem}
\left(\vert \varphi'(t)\vert^{p-2} \varphi'(t)\right)'+\mu\, \vert \varphi(t)\vert^{p-2}\varphi(t)=0, \quad \varphi(0)=0, \quad \varphi'(D)=0.
\end{equation}
In the case where $p=2$,
we see $\mu_{2,D}= \pi^{2}(2D)^{-2}$.

For a positive number $D \in (0,\infty)$,
and for a connected component $\bm_{1}$ of $\bm$,
we denote by $[0,D]\times_{F} \bm_{1}$ the twisted product space $([0,D]\times \bm_{1},dt^{2}+F^{2}_{x}(t)\,h)$,
where for every $x\in \bm_{1}$
the function $F_{x}:[0,D] \to (0,\infty)$ is defined as (\ref{eq:warping function}).
The \textit{inscribed radius of $M$} is defined as
\begin{equation*}
\dm:=\sup_{p\in M} \rho_{\bm}(p).
\end{equation*}
Suppose that
$M$ is compact.
We say that
the metric space $(M,d_{M})$ is an \textit{$F$-model space}
if $M$ is isometric to either
(1) for a connected component $\bm_{1}$ of $\bm$,
the twisted product space $[0,2\dm]\times_{F} \bm_{1}$;
or (2) for an involutive isometry $\sigma$ of $\bm$ without fixed points,
the quotient space $([0,2\dm]\times_{F} \bm)/G_{\sigma}$,
where $G_{\sigma}$ is the isometry group on $[0,2\dm]\times_{F} \bm$ of the identity and the involute isometry $\hat{\sigma}$ defined by $\hat{\sigma}(t,x):=(2\dm-t,\sigma(x))$.
If $(M,d_{M})$ is an $F$-model space,
and if for every $x\in \bm$
the function $F_{x}$ is equal to $1$ on $[0,\dm]$,
then we call the $F$-model space $(M,d_{M})$ an \textit{equational model space}.
The notion of the equational model spaces coincides with that of the $(0,0)$-equational model spaces introduced in \cite{Sa2}.

We prove the following rigidity theorem for $\mu_{f,1,p}$:
\begin{thm}\label{thm:eigenvalue rigidity}
Let $M$ be a connected complete Riemannian manifold with boundary,
and let $f:M\to \mathbb{R}$ be a smooth function.
Suppose that
$M$ is compact.
Let $p\in (1,\infty)$,
and let $\kappa \in \mathbb{R}$ and $\lambda \in \mathbb{R}$ satisfy the subharmonic-condition.
For $N \in (-\infty,1]$
we suppose $\ric^{N}_{f,M}\geq \kappa$,
and $H_{f,\bm} \geq \lambda$.
For $D\in (0,\infty)$
we assume $\dm \leq D$.
Then
\begin{equation}\label{eq:eigenvalue rigidity}
\mu_{f,1,p}(M)\geq \mu_{p,D}.
\end{equation}
If the equality in $(\ref{eq:eigenvalue rigidity})$ holds,
then $\dm=D$,
and the metric space $(M,d_{M})$ is an $F$-model space.
Moreover,
if $N\in (-\infty,1)$,
then for every $x\in \bm$
the function $f\circ \gamma_{x}$ is constant on $[0,D]$;
in particular,
$(M,d_{M})$ is an equational model space.
\end{thm}
In the unweighted case,
Li and Yau \cite{LY} have obtained the estimate (\ref{eq:eigenvalue rigidity}),
and Kasue \cite{K3} has proved Theorem \ref{thm:eigenvalue rigidity} when $p=2,\kappa=0$ and $\lambda=0$.
In \cite{Sa2},
the author has proved Theorem \ref{thm:splitting theorem} in the usual weighted case when $\kappa=0$ and $\lambda=0$.
\begin{rem}
We do not know whether
Theorem \ref{thm:eigenvalue rigidity} can be extended to the weighted case of $N \in (1,n)$.
\end{rem}
Suppose that
$M$ is compact.
We say that
the metric space $(M,d_{M})$ is a \textit{warped model space} if
there exist a function $f_{0}:[0,2\dm] \to \mathbb{R}$,
and a Riemannian metric $h_{0}$ on $\bm$ such that
$M$ is isometric to either
(1) for a connected component $\bm_{1}$ of $\bm$,
the warped product space $([0,2\dm] \times \bm_{1},dt^{2}+e^{2\frac{f_{0}(t)}{n-1}}h_{0})$;
or (2) for an involutive isometry $\sigma$ of $\bm$ without fixed points,
the quotient space $([0,2\dm]\times \bm,dt^{2}+e^{2\frac{f_{0}(t)}{n-1}}h_{0})/G_{\sigma}$,
where $G_{\sigma}$ is the isometry group on $([0,2\dm]\times \bm,dt^{2}+e^{2\frac{f_{0}(t)}{n-1}}h_{0})$ of the identity and the involute isometry $\hat{\sigma}$ defined as $\hat{\sigma}(t,x):=(2\dm-t,\sigma(x))$.

We obtain the following corollary of Theorem \ref{thm:eigenvalue rigidity}:
\begin{cor}\label{cor:warped eigenvalue rigidity}
Under the same setting as in Theorem \ref{thm:eigenvalue rigidity},
if $N=1$ and $\kappa=0$,
and if the equality in $(\ref{eq:eigenvalue rigidity})$ holds,
then the metric space $(M,d_{M})$ is a warped model space.
\end{cor}
\subsection{Organization}\label{sec:Organization}
In Section \ref{sec:Preliminaries},
we prepare some notations
and recall the basic facts for Riemannian manifolds with boundary.
In Section \ref{sec:Comparison theorems},
we show Laplacian comparison results for $\rho_{\bm}$.
In Section \ref{sec:Volume comparisons},
we show volume comparison results,
and conclude Theorem \ref{thm:volume growth rigidity} and Corollary \ref{cor:warped volume growth rigidity}.
In Section \ref{sec:Splitting theorems},
we prove Theorem \ref{thm:splitting theorem} and Corollary \ref{cor:warped splitting theorem},
and discuss its variants.
In Section \ref{sec:Eigenvalue rigidity},
we prove Theorem \ref{thm:eigenvalue rigidity} and Corollary \ref{cor:warped eigenvalue rigidity}.
We also obtain an explicit lower bound for $\mu_{f,1,p}$ (see Subsection \ref{sec:Eigenvalue estimates}).

\subsection*{{\rm Acknowledgements}}\label{sec:Acknowledgements}
The author would like to express his gratitude to Professor Koichi Nagano for his constant advice and suggestions.
The author would also like to thank Professor Shin-ichi Ohta for his valuable comments.
The author would like to thank Professor William Wylie for his valuable advice concerning Proposition \ref{prop:twisted to warped}.

%% file: section2.tex
\section{Preliminaries}\label{sec:Preliminaries}
We refer to \cite{S} for the basics of Riemannian manifolds with boundary (cf. Section 2 in \cite{Sa1}, and in \cite{Sa2}).
\subsection{Riemannian manifolds with boundary}
For $n\geq 2$, 
let $M$ be an $n$-dimensional,
connected Riemannian manifold with (smooth) boundary
with Riemannian metric $g$.
For a point $p\in \inte M$, 
let $T_{p}M$ be the tangent space at $p$ on $M$,
and let $U_{p}M$ be the unit tangent sphere at $p$ on $M$.
We denote by $\Vert \cdot \Vert$ the standard norm induced from $g$.
If $v_{1},\dots,v_{k}\in T_{p}M$ are linearly independent,
then it holds that $\Vert v_{1}\wedge \cdots \wedge v_{k} \Vert=\sqrt{\det (g(v_{i},v_{j}))}$.

Let $d_{M}$ be the Riemannian distance on $M$
induced from the length structure determined by $g$.
For an interval $I$,
we say that
a curve $\gamma:I\to M$ is a \textit{normal minimal geodesic}
if for all $s,t\in I$
we have $d_{M}(\gamma(s),\gamma(t))=\vert s-t\vert$,
and $\gamma$ is a \textit{normal geodesic}
if for each $t\in I$
there exists an interval $J\subset I$ with $t\in J$ such that $\gamma|_{J}$ is a normal minimal geodesic.
If $M$ is complete with respect to $d_{M}$,
then the Hopf-Rinow theorem for length spaces (see e.g., Theorem 2.5.23 in \cite{BBI}) tells us that
the metric space $(M,d_{M})$ is a proper,
geodesic space;
namely,
all closed bounded subsets of $M$ are compact,
and 
for every pair of points in $M$
there exists a normal minimal geodesic connecting them.

For $i=1,2$,
let $M_{i}$ be connected Riemannian manifolds with boundary with Riemannian metric $g_{i}$.
For each $i$,
the boundary $\bm_{i}$ carries the induced Riemannian metric $h_{i}$.
We say that a homeomorphism $\Phi:M_{1}\to M_{2}$ is a \textit{Riemannian isometry with boundary} from $M_{1}$ to $M_{2}$ if $\Phi$ satisfies the following conditions:
\begin{enumerate}
 \item $\Phi|_{\inte M_{1}}:\inte M_{1} \to \inte M_{2}$ is smooth, and $(\Phi|_{\inte M_{1}})^{\ast} (g_{2})=g_{1}$;\label{enum:inner isom}
 \item $\Phi|_{\bm_{1}}:\bm_{1} \to \bm_{2}$ is smooth, and $(\Phi|_{\bm_{1}})^{\ast} (h_{2})=h_{1}$.\label{enum:bdry isom}
\end{enumerate}
If $\Phi:M_{1}\to M_{2}$ is a Riemannian isometry with boundary,
then the inverse $\Phi^{-1}$ is also a Riemannian isometry with boundary.
Notice that
there exists a Riemannian isometry with boundary from $M_{1}$ to $M_{2}$ if and only if
the metric space $(M_{1},d_{M_{1}})$ is isometric to $(M_{2},d_{M_{2}})$ (see e.g., Section 2 in \cite{Sa1}).

\subsection{Jacobi fields orthogonal to the boundary}\label{sec:Jacobi fields orthogonal to the boundary}
Let $M$ be a connected Riemannian manifold with boundary with Riemannian metric $g$.
For a point $x\in \bm$,
and for the tangent space $T_{x}\bm$ at $x$ on $\bm$,
let $T_{x}^{\perp} \bm$ be the orthogonal complement of $T_{x}\bm$ in the tangent space at $x$ on $M$.
Take $u\in T_{x}^{\perp}\bm$.
For the second fundamental form $S$ of $\bm$,
let $A_{u}:T_{x}\bm \to T_{x}\bm$ be the \textit{shape operator} for $u$ defined as
\begin{equation*}
g(A_{u}v,w):=g(S(v,w),u).
\end{equation*}
We denote by $u_{x}$ the unit inner normal vector at $x$.
The \textit{mean curvature} $H_{x}$ at $x$ is defined as $H_{x}:=\tr A_{u_{x}}$.
We denote by $\gamma_{x}:[0,T)\to M$ the normal geodesic with initial conditions $\gamma_{x}(0)=x$ and $\gamma_{x}'(0)=u_{x}$.
We say that a Jacobi field $Y$ along $\gamma_{x}$ is a $\bm$-\textit{Jacobi field} if $Y$ satisfies the following initial conditions:
\begin{equation*}
Y(0)\in T_{x}\bm, \quad Y'(0)+A_{u_{x}}Y(0)\in T_{x}^{\perp}\bm.
\end{equation*}
We say that $\gamma_{x}(t_{0})$ is a \textit{conjugate point} of $\bm$ along $\gamma_{x}$
if there exists a non-zero $\bm$-Jacobi field $Y$ along $\gamma_{x}$ with $Y(t_{0})=0$.
We denote by $\tau_{1}(x)$ the first conjugate value for $\bm$ along $\gamma_{x}$.
It is well-known that for all $x\in \bm$ and $t>\tau_{1}(x)$,
we have $t>\rho_{\bm}(\gamma_{x}(t))$.

For the normal tangent bundle $\tbp:=\bigcup_{x\in \bm} T_{x}^{\perp}\bm$ of $\bm$,
let $0(\tbp)$ be the zero-section $\bigcup_{x\in \bm} \{\,0_{x}\in T_{x}^{\perp}\bm\, \}$ of $T^{\perp}\bm$.
On an open neighborhood of $0(\tbp)$ in $\tbp$, 
the normal exponential map $\expp$ of $\bm$ is defined as $\expp(x,u):=\gamma_{x}(\Vert u\Vert)$
for $x\in \bm$ and $u\in T_{x}^{\perp}\bm$.

For $x\in \bm$ and $t\in [0,\tau_{1}(x))$,
we denote by $\theta(t,x)$ the absolute value of the Jacobian of $\expp$ at $(x,tu_{x})\in \tbp$.
For each $x\in \bm$,
we choose an orthonomal basis $\{e_{x,i}\}_{i=1}^{n-1}$ of $T_{x}\bm$.
For each $i$,
let $Y_{x,i}$ be the $\bm$-Jacobi field along $\gamma_{x}$ with initial conditions $Y_{x,i}(0)=e_{x,i}$ and $Y'_{x,i}(0)=-A_{u_{x}}e_{x,i}$.
Note that for all $x\in \bm$ and $t\in [0,\tau_{1}(x))$,
we have $\theta(t,x)=\Vert Y_{x,1}(t)\wedge \cdots \wedge Y_{x,n-1}(t)\Vert$.
This does not depend on the choice of the orthonormal bases.

\subsection{Cut locus for the boundary}\label{sec:Cut locus for the boundary}
We recall the basic properties of the cut locus for the boundary.
The basic properties seem to be well-known.
We refer to \cite{Sa1} for the proofs.

Let $M$ be a connected complete Riemannian manifold with boundary with Riemannian metric $g$.
For $p\in M$, 
we call $x\in \bm$ a \textit{foot point} on $\bm$ of $p$ if $d_{M}(p,x)=\rho_{\bm}(p)$.
Since $(M,d_{M})$ is proper, 
every point in $M$ has at least one foot point on $\bm$.
For $p\in \inte M$, 
let $x\in \bm$ be a foot point on $\bm$ of $p$.
Then there exists a unique normal minimal geodesic $\gamma:[0,l]\to M$ from $x$ to $p$
such that $\gamma=\gamma_{x}|_{[0,l]}$,
where $l=\rho_{\bm}(p)$.
In particular,
$\gamma'(0)=u_{x}$ and $\gamma|_{(0,l]}$ lies in $\inte M$.

Let $\tau:\bm\to \mathbb{R}\cup \{\infty\}$ be the function defined as (\ref{eq:cut point}).
By the property of $\tau_{1}$,
for all $x\in \bm$
we have $0<\tau(x)\leq \tau_{1}(x)$.
For the inscribed radius $D(M,\bm)$ of $M$,
from the definition of $\tau$,
we have $D(M,\bm)=\sup_{x\in \bm} \tau(x)$.
The function $\tau$ is continuous on $\bm$.

The continuity of $\tau$ implies the following (see e.g., Section 3 in \cite{Sa1}):
\begin{lem}\label{lem:bmcompact}
Suppose that
$\bm$ is compact.
Then $D(M,\bm)$ is finite if and only if $M$ is compact.
\end{lem}

We put
\begin{align*}
TD_{\bm}  &:= \bigcup_{x\in \bm} \{\,t\,u_{x} \in T^{\perp}_{x}\bm \mid t\in[0,\tau(x)) \,\},\\
T\cut \bm &:= \bigcup_{x\in \bm} \{\,\tau(x)\,u_{x}\in T^{\perp}_{x}\bm \mid \tau(x)<\infty \,\},
\end{align*}
and define $D_{\bm}:=\expp (TD_{\bm})$ and $\cut \bm:=\expp (T\cut \bm)$.
We call $\cut \bm$ the \textit{cut locus for the boundary} $\bm$.
From the continuity of $\tau$,
the set $\cut \bm$ is a null set of $M$.
Furthermore,
we have
\begin{equation*}
\inte M=(D_{\bm}\setminus \bm) \sqcup \cut \bm,\quad M=D_{\bm}\sqcup \cut \bm.
\end{equation*}
This implies that
if $\cut\bm=\emptyset$,
then $\bm$ is connected.
The set $TD_{\bm}\setminus 0(T^{\perp}\bm)$ is a maximal domain in $T^{\perp}\bm$ on which $\expp$ is regular and injective.

In \cite{Sa2},
we have already known the following:
\begin{lem}\label{lem:splitting lemma}
If there exists a connected component $\bm_{0}$ of $\bm$ such that 
for all $x\in \bm_{0}$ we have $\tau(x)=\infty$,
then $\bm$ is connected and $\cut \bm=\emptyset$.
\end{lem}
The function $\rho_{\bm}$ is smooth on $\inte M\setminus \cut \bm$.
For each $p\in \inte M\setminus \cut \bm$,
the gradient vector $\nabla \rho_{\bm}(p)$ of $\rho_{\bm}$ at $p$
is given by $\nabla \rho_{\bm}(p)=\gamma'(l)$,
where $\gamma:[0,l]\to M$ is the normal minimal geodesic from the foot point on $\bm$ of $p$ to $p$.

For $\Omega \subset M$,
we denote by $\bar{\Omega}$ the closure of $\Omega$ in $M$,
and by $\partial \Omega$ the boundary of $\Omega$ in $M$.
For a domain $\Omega$ in $M$ such that
$\partial \Omega$ is a smooth hypersurface in $M$,
we denote by $\vol_{\partial \Omega}$ the canonical Riemannian volume measure on $\partial \Omega$.

We have the following fact to avoid the cut locus for the boundary
that has been stated in \cite{Sa2} (see Lemma 2.6 in \cite{Sa2}):
\begin{lem}\label{lem:avoiding the cut locus}
Let $\Omega$ be a domain in $M$ such that
$\partial \Omega$ is a smooth hypersurface in $M$.
Then there exists a sequence $\{\Omega_{k}\}_{k\in \mathbb{N}}$ of closed subsets of $\bar{\Omega}$ such that
for every $k\in \mathbb{N}$,
the set $\partial \Omega_{k}$ is a smooth hypersurface in $M$ except for a null set in $(\partial \Omega,\vol_{\partial \Omega})$
satisfying the following properties:
\begin{enumerate}
\item for all $k_{1},k_{2}\in \mathbb{N}$ with $k_{1}<k_{2}$,
         we have $\Omega_{k_{1}}\subset \Omega_{k_{2}}$;
\item $\bar{\Omega} \setminus \cut \bm=\bigcup_{k\in \mathbb{N}}\,\Omega_{k}$;
\item for every $k\in \mathbb{N}$,
         and for almost every point $p \in \partial \Omega_{k}\cap \partial \Omega$ in $(\partial \Omega,\vol_{\partial \Omega})$,
         there exists the unit outer normal vector for $\Omega_{k}$ at $p$
         that coincides with the unit outer normal vector on $\partial \Omega$ for $\Omega$ at $p$;
\item for every $k\in \mathbb{N}$,
         on $\partial \Omega_{k}\setminus \partial \Omega$,
         there exists the unit outer normal vector field $\nu_{k}$ for $\Omega_{k}$ such that $g(\nu_{k},\nabla \rho_{\bm})\geq 0$.
\end{enumerate}
Moreover,
if $\bar{\Omega}=M$,
then for every $k\in \mathbb{N}$,
the set $\partial \Omega_{k}$ is a smooth hypersurface in $M$,
and satisfies $\partial \Omega_{k}\cap \bm=\bm$.
\end{lem}
As noticed in \cite{Sa2},
for the cut locus for a single point,
we have known a similar fact to Lemma \ref{lem:avoiding the cut locus} (see e.g., Theorem 4.1 in \cite{Che2}).
One can prove Lemma \ref{lem:avoiding the cut locus}
by a similar method to the case of the cut locus for a single point.

\subsection{Busemann functions and asymptotes}\label{subsec:Busemann functions and asymptotes}
Let $M$ be a connected complete Riemannian manifold with boundary.
A normal geodesic $\gamma:[0,\infty)\to M$ is said to be a \textit{ray}
if for all $s,t\in [0,\infty)$
it holds that $d_{M}(\gamma(s),\gamma(t))=|s-t|$.
For a ray $\gamma:[0,\infty)\to M$, 
the \textit{Busemann function} $b_{\gamma}:M\to \mathbb{R}$ \textit{of} $\gamma$ is defined as
\begin{equation*}
b_{\gamma}(p):=\lim_{t\to \infty}(t-d_{M}(p,\gamma(t))).
\end{equation*}

Take a ray $\gamma:[0,\infty)\to M$
and a point $p\in \inte M$,
and choose a sequence $\{t_{i}\}$ with $t_{i}\to \infty$.
For each $i$,
we take a normal minimal geodesic $\gamma_{i}:[0,l_{i}]\to M$ from $p$ to $\gamma(t_{i})$.
Since $\gamma$ is a ray,
it follows that $l_{i}\to \infty$.
Take a sequence $\{T_{j}\}$ with $T_{j}\to \infty$.
Using the fact that $M$ is proper,
we take a subsequence $\{\gamma_{1,i}\}$ of $\{\gamma_{i}\}$,
and a normal minimal geodesic $\gamma_{p,1}:[0,T_{1}]\to M$ from $p$ to $\gamma_{p,1}(T_{1})$
such that $\gamma_{1,i}|_{[0,T_{1}]}$ uniformly converges to $\gamma_{p,1}$.
In this manner,
take a subsequence $\{\gamma_{2,i}\}$ of $\{\gamma_{1,i}\}$
and a normal minimal geodesic $\gamma_{p,2}:[0,T_{2}]\to M$ from $p$ to $\gamma_{p,2}(T_{2})$
such that $\gamma_{2,i}|_{[0,T_{2}]}$ uniformly converges to $\gamma_{p,2}$,
where $\gamma_{p,2}|_{[0,T_{1}]}=\gamma_{p,1}$.
By means of a diagonal argument,
we obtain a subsequence $\{\gamma_{k}\}$ of $\{\gamma_{i}\}$
and a ray $\gamma_{p}$ in $M$ such that
for every $t\in (0,\infty)$
we have $\gamma_{k}(t)\to \gamma_{p}(t)$ as $k\to \infty$.
We call such a ray $\gamma_{p}$ an \textit{asymptote for} $\gamma$ \textit{from} $p$.

The following lemmas have been shown in \cite{Sa1}.
\begin{lem}\label{lem:busemann function}
Suppose that for some $x\in \bm$
we have $\tau(x)=\infty$.
Take $p\in \inte M$.
If $b_{\gamma_{x}}(p)=\rho_{\bm}(p)$,
then $p\notin \cut \bm$.
Moreover,
for the unique foot point $y$ on $\bm$ of $p$,
we have $\tau(y)=\infty$.
\end{lem}
\begin{lem}\label{lem:asymptote}
Suppose that for some $x\in \bm$
we have $\tau(x)=\infty$.
For $l \in (0,\infty)$,
put $p:=\gamma_{x}(l)$.
Then there exists $\epsilon \in (0,\infty)$ such that
for all $q\in B_{\epsilon}(p)$,
all asymptotes for the ray $\gamma_{x}$ from $q$ lie in $\inte M$.
\end{lem}

\subsection{Weighted Riemannian manifolds with boundary}
Let $M$ be a connected complete Riemannian manifold with boundary,
and let $f:M\to \mathbb{R}$ be a smooth function.
For a smooth function $\varphi$ on $M$,
the \textit{weighted Laplacian $\Delta_{f}\varphi$ for $\varphi$} is defined by
\begin{equation*}\label{eq:weighted Laplacian}
\Delta_{f} \varphi:=\Delta \varphi+g(\nabla f,\nabla \varphi),
\end{equation*}
where $\Delta \varphi$ is the Laplacian for $\varphi$ defined as the minus of the trace of its Hessian.
Note that
$\Delta_{f}$ coincides with the $(f,2)$-Laplacian $\Delta_{f,2}$.

It seems that
the following formula of Bochner type is well-known (see \cite{Lic}, and Chapter 14 in \cite{V}).
\begin{prop}[\cite{Lic}]\label{prop:Bochner formula}
For every smooth function $\varphi$ on $M$,
we have
\begin{equation*}
-\frac{1}{2}\,\Delta_{f}\,\Vert \nabla \varphi \Vert^{2}=\ric^{\infty}_{f}(\nabla \varphi)+\Vert \Hess \varphi \Vert^{2}-g\left(\nabla \Delta_{f}\,\varphi,\nabla \varphi \right),
\end{equation*}
where $\Vert \Hess \varphi \Vert$ is the Hilbert-Schmidt norm of $\Hess \varphi$.
\end{prop}
Notice that
for every $x \in \bm$,
and for every $t \in (0,\tau(x))$,
the value $\Delta \rho_{\bm}(\gamma_{x}(t))$ is equal to the mean curvature $H_{x,t}$ of the $t$-level set of $\rho_{\bm}$ at $\gamma_{x}(t)$ toward $\nabla \rho_{\bm}$.
In our weighted case,
by the definition of the weighted Laplacian,
we see the following:
\begin{lem}\label{lem:Laplacian and mean curvature}
Take $x\in \bm$.
Then for every $t \in (0,\tau(x))$,
the value $\Delta_{f}\rho_{\bm}(\gamma_{x}(t))$ is equal to
the $f$-mean curvature $H_{f,x,t}$ of the $t$-level set of $\rho_{\bm}$ at $\gamma_{x}(t)$ toward $\nabla \rho_{\bm}$ defined as
\begin{equation*}
H_{f,x,t}:=H_{x,t}+g(\nabla f,\nabla \rho_{\bm})(\gamma_{x}(t)).
\end{equation*}
In particular,
$\Delta_{f}\rho_{\bm}(\gamma_{x}(t))$ tends to $H_{f,x}$ as $t \to 0$,
where $H_{f,x}$ is the $f$-mean curvature of $\bm$ at $x$ defined as  $(\ref{eq:weighted mean curvature})$.
\end{lem}
For $x \in \bm$ and $t \in [0,\tau(x))$,
we put
\begin{equation}\label{eq:weighted volume element}
\theta_{f}(t,x):=e^{-f(\gamma_{x}(t))}\, \theta(t,x),
\end{equation}
where $\theta(t,x)$ is the absolute value of the Jacobian of the map $\expp$ at $(x,tu_{x})\in \tbp$.
For all $x\in \bm$ and $t\in (0,\tau(x))$,
it holds that
\begin{equation}\label{eq:Laplacian representation}
\Delta_{f}\, \rho_{\bm}(\gamma_{x}(t)) =-(\log \theta(t,x))'+f(\gamma_{x}(t))'=-\frac{\theta_{f}'(t,x)}{\theta_{f}(t,x)}.
\end{equation}
Let $\bar{\theta}_{f}:[0,\infty) \times \bm \to \mathbb{R}$ be a function defined by
\begin{equation}\label{eq:extend volume element}
  \bar{\theta}_{f}(t,x) := \begin{cases}
                            \theta_{f}(t,x) & \text{if $t< \tau(x)$}, \\
                            0           & \text{if $t \geq \tau(x)$}.
                       \end{cases}
\end{equation}

The following has been shown in \cite{Sa2}:
\begin{lem}\label{lem:Basic volume lemma}
If $\bm$ is compact,
then for all $r\in (0,\infty)$
\begin{equation*}\label{eq:volume comparison1}
m_{f}( B_{r}(\bm))=\int_{\bm} \int^{r}_{0}\bar{\theta}_{f}(t,x)\,dt\,d\vol_{h},
\end{equation*}
where $m_{f}$ denotes the weighted measure on $M$ defined as $(\ref{eq:weighted measure})$,
and $h$ denotes the induced Riemannian metric on $\bm$.
\end{lem}

\subsection{Twisted and warped product spaces}
In \cite{W},
for the proof of a splitting theorem of Cheeger-Gromoll type,
Wylie has proved that
a twisted product space over $\mathbb{R}$ becomes a warped product space
under a non-negativity of the $1$-weighted Ricci curvature (see Proposition 2.2 in \cite{W}).
The proof is based on a pointwise calculation,
and the same holds true for a twisted product space over an arbitrary interval.

From the argument in the proof of Proposition 2.2 in \cite{W},
we can derive the following in our setting:
\begin{prop}[\cite{W}]\label{prop:twisted to warped}
Let $M$ be a Riemannian manifold with boundary,
and let $f:M\to \mathbb{R}$ be a smooth function.
Suppose that
there exist an interval $I$ in the form of $[0,\infty)$
or $[0,D]$ for some positive number $D$,
and a connected component $\bm_{1}$ of $\bm$ such that
$M$ is isometric to a twisted product space $I \times_{F} \bm_{1}$.
If $\ric^{1}_{f,M} \geq 0$,
then there exist functions $f_{0}:I \to \mathbb{R}$ and $f_{1}:\bm_{1}\to \mathbb{R}$ such that
for all $t \in I$ and $x \in \bm_{1}$
we have $f(\gamma_{x}(t))=f_{0}(t)+f_{1}(x)$;
in particular,
$M$ is isometric to a warped product space $(I\times \bm_{1},dt^{2}+e^{\frac{2 f_{0}(t)}{n-1}}h_{0})$,
where for the induced metric $h$ on $\bm_{1}$
we put $h_{0}:=e^{2\frac{f_{1}-(f|_{\bm_{1}})}{n-1}}h$.
\end{prop}
Notice that
Proposition \ref{prop:twisted to warped} has been implicitly used in the proof of Theorem 5.1 in \cite{W}.

%% file: section3.tex
\section{Laplacian comparisons}\label{sec:Comparison theorems}
In this section,
let $M$ be an $n$-dimensional,
connected complete Riemannian manifold with boundary with Riemannian metric $g$,
and let $f:M\to \mathbb{R}$ be a smooth function.
\subsection{Basic comparisons}
Recall that
for $x\in \bm$,
the function $F_{x}:[0,\tau(x)]\setminus \{\infty\} \to (0,\infty)$ is defined as (\ref{eq:warping function}).
For $\kappa,\lambda \in \mathbb{R}$,
we define a function $F_{\kappa,\lambda,x}:[0,\tau(x)]\setminus \{\infty\} \to \mathbb{R}$ by
\begin{equation}\label{eq:subharmonic warping function}
F_{\kappa,\lambda,x}(t):=\kappa \int^{t}_{0}\,F^{2}_{x}(s)\,ds+\lambda.
\end{equation}
Note that
if $\kappa$ and $\lambda$ satisfy the subharmonic-condition,
then for every $x\in \bm$
the function $F_{\kappa,\lambda,x}$ is non-negative.

We have the following Laplacian comparison inequality for $\rho_{\bm}$:
\begin{lem}\label{lem:Laplacian comparison}
Take $x\in \bm$.
For $\kappa,\lambda \in \mathbb{R}$
and for $N\in (-\infty,1]$
we suppose that for all $t\in (0,\tau(x))$
we have $\ric^{N}_{f}(\gamma'_{x}(t)) \geq \kappa$,
and suppose $H_{f,x}\geq \lambda$.
Then for all $t\in (0,\tau(x))$
we have
\begin{equation}\label{eq:Laplacian comparison}
\Delta_{f}\, \rho_{\bm} (\gamma_{x}(t)) \geq F^{-2}_{x}(t)\,F_{\kappa,\lambda,x}(t).
\end{equation}
In particular,
if $\kappa$ and $\lambda$ satisfy the subharmonic-condition,
then for all $t\in (0,\tau(x))$
we have $\Delta_{f}\, \rho_{\bm} (\gamma_{x}(t)) \geq 0$.
\end{lem}
\begin{proof}
The function $\rho_{\bm} \circ \gamma_{x}$ is smooth on $(0,\tau(x))$.
We put $h_{f,x}:=\left(\Delta_{f}\,\rho_{\bm}\right)\circ \gamma_{x}$.
We first show that
for all $s \in (0,\tau(x))$
\begin{equation}\label{eq:monotonicity}
\left(F^{2}_{x}(s)\,h_{f,x}(s)-\kappa\,\int^{s}_{0}\,F^{2}_{x}(u)\,du\right)' \geq 0.
\end{equation}
Fix $s\in (0,\tau(x))$,
and put $f_{x}:=f \circ \gamma_{x}$.
We apply Proposition \ref{prop:Bochner formula} to the function $\rho_{\bm}$.
Since $\Vert \nabla \rho_{\bm} \Vert=1$ along $\gamma_{x}|_{(0,\tau(x))}$,
it holds that
\begin{align*}\label{eq:Bochner formula1}
0 &= \ric^{\infty}_{f}(\gamma'_{x}(s))+\Vert \Hess \rho_{\bm}\Vert^{2}\left(\gamma_{x}(s)\right)-g\left(\nabla \Delta_{f}\rho_{\bm},\nabla \rho_{\bm}  \right)(\gamma_{x}(s))\\
   &= \left(\ric^{N}_{f}(\gamma'_{x}(s))+\frac{f'_{x}(s)^{2}}{N-n}\right)+\Vert \Hess \rho_{\bm}\Vert^{2}\left(\gamma_{x}(s)\right)-h'_{f,x}(s). \notag
\end{align*}
From the assumption $\ric^{N}_{f}(\gamma'_{x}(s)) \geq \kappa$,
we deduce
\begin{equation}\label{eq:curvature assumption}
0 \geq \kappa+\frac{f'_{x}(s)^{2}}{N-n}+\Vert \Hess \rho_{\bm}\Vert^{2}\left(\gamma_{x}(s)\right)-h'_{f,x}(s).
\end{equation}
By the Cauchy-Schwarz inequality,
we have
\begin{equation}\label{eq:Cauchy-Schwarz inequality}
\Vert \Hess \rho_{\bm}\Vert^{2}\left(\gamma_{x}(s)\right) \geq \frac{\left(\Delta \rho_{\bm}(\gamma_{x}(s))\right)^{2}}{n-1}=\frac{\left( h_{f,x}(s)-f'_{x}(s)\right)^{2}}{n-1}.
\end{equation}
Combining (\ref{eq:curvature assumption}) and (\ref{eq:Cauchy-Schwarz inequality}),
we see
\begin{align}\label{eq:combination}\notag
0 &\geq \kappa+\frac{f'_{x}(s)^{2}}{N-n}+\frac{\left(h_{f,x}(s)-f'_{x}(s)\right)^{2}}{n-1}-h'_{f,x}(s)\\
   &   =   \kappa+\frac{\left(1-N\right)f'_{x}(s)^{2}}{\left(n-1\right)\left(n-N\right)}+\frac{h_{f,x}(s)^{2}}{n-1}-\left(\frac{2h_{f,x}(s)f'_{x}(s)}{n-1}+h'_{f,x}(s)\right).
\end{align}
In the right hand side of the equation (\ref{eq:combination}),
by $N\leq 1$,
the second term is non-negative.
The third one is non-negative.
The last one satisfies
\begin{align*}
\frac{2h_{f,x}(s)f'_{x}(s)}{n-1}+h'_{f,x}(s)
&=e^{-\frac{2f_{x}(s)}{n-1}}\,\left(e^{\frac{2f_{x}(s)}{n-1}}\,h_{f,x}(s)\right)'\\
&=F^{-2}_{x}(s)\,\left(F^{2}_{x}(s)\,h_{f,x}(s)\right)'.
\end{align*}
Hence,
we have $0\geq \kappa-F^{-2}_{x}(s)\,\left(F^{2}_{x}(s)\,h_{f,x}(s)\right)'$.
This implies (\ref{eq:monotonicity}).

We see that
$F_{x}(s)$ tends to $1$ as $s\to 0$.
Furthermore,
by Lemma \ref{lem:Laplacian and mean curvature},
$h_{f,x}(s)$ tends to $H_{f,x}$ as $s\to 0$.
It follows that
\begin{equation}\label{eq:initial value}
F^{2}_{x}(s)\,h_{f,x}(s)-\kappa\,\int^{s}_{0}\,F^{2}_{x}(u)\,du \to H_{f,x}
\end{equation}
as $s\to 0$.
By (\ref{eq:monotonicity}) and (\ref{eq:initial value}),
for all $s,t\in (0,\tau(x))$ with $s\leq t$
\begin{multline}\label{eq:Bochner formula4}
        F^{2}_{x}(t)\,h_{f,x}(t)-\kappa\,\int^{t}_{0}\,F^{2}_{x}(u)\,du\\
\geq F^{2}_{x}(s)\,h_{f,x}(s)-\kappa\,\int^{s}_{0}\,F^{2}_{x}(u)\,du
\geq H_{f,x}
\geq \lambda.
\end{multline}
Thus,
we arrive at (\ref{eq:Laplacian comparison}).
\end{proof}
\begin{rem}
Under the same setting as in Lemma \ref{lem:Laplacian comparison},
Wylie \cite{W} has shown a Laplacian comparison inequality for the distance function from a connected component of the boundary 
that is similar to (\ref{eq:Laplacian comparison})
when $\kappa=0$ and $\lambda=0$ (see the proof of Theorem 5.1 in \cite{W}).
\end{rem}
\begin{rem}\label{rem:equality in the lemma}
Assume that
for some $t_{0}\in (0,\tau(x))$
the equality in $(\ref{eq:Laplacian comparison})$ holds.
Then (\ref{eq:Bochner formula4}) implies that
$F^{2}_{x}\,h_{f,x}=F_{\kappa,\lambda,x}$ and 
$\left(F^{2}_{x}\,h_{f,x}  \right)'=\kappa\, F^{2}_{x}$ on $(0,t_{0})$.
Hence,
for every $t \in (0,t_{0})$,
the equality in the Cauchy-Schwarz inequality in (\ref{eq:Cauchy-Schwarz inequality}) holds;
in particular,
there exists a function $\varphi$ on $\gamma_{x}((0,t_{0}))$ such that
at each point on $\gamma_{x}((0,t_{0}))$
we have $\Hess \rho_{\bm}=\varphi\,g$ on the orthogonal complement of $\nabla \rho_{\bm}$.
Furthermore,
for every $t \in (0,t_{0})$,
the second term and the third one in the right hand side of (\ref{eq:combination}) are equal to $0$;
in particular,
$(1-N)(f'_{x})^{2}=0$ on $(0,t_{0})$.
\end{rem}
From Lemma \ref{lem:Laplacian comparison},
we derive the following:
\begin{lem}\label{lem:Basic comparison}
Take $x\in \bm$.
For $\kappa,\lambda \in \mathbb{R}$ and for $N\in (-\infty,1]$
we suppose that for all $t\in (0,\tau(x))$
we have $\ric^{N}_{f}(\gamma'_{x}(t)) \geq \kappa$,
and suppose $H_{f,x}\geq \lambda$.
Then for all $s,t\in [0,\tau(x))$ with $s\leq t$
we have
\begin{equation*}
\theta_{f}(t,x) \leq e^{-\int^{t}_{s}\,F^{-2}_{x}(u)\,F_{\kappa,\lambda,x}(u)\,du}\,\theta_{f}(s,x),
\end{equation*}
where $\theta_{f}(t,x)$ is defined as $(\ref{eq:weighted volume element})$.
In particular,
if $\kappa$ and $\lambda$ satisfy the subharmonic-condition,
then for all $s,t\in [0,\tau(x))$ with $s\leq t$
we have $\theta_{f}(t,x) \leq \theta_{f}(s,x)$.
\end{lem}
\begin{proof}
By (\ref{eq:Laplacian representation}) and Lemma \ref{lem:Laplacian comparison},
for all $t\in (0,\tau(x))$
\begin{equation*}
\frac{d}{dt} \log \frac{e^{-\int^{t}_{0}\,F^{-2}_{x}(u)F_{\kappa,\lambda,x}(u)\,du}}{\theta_{f}(t,x)}\\
=-F^{-2}_{x}(t)F_{\kappa,\lambda,x}(t)+\Delta_{f}\,\rho_{\bm}(\gamma_{x}(t))\geq 0.
\end{equation*}
It follows that
for all $s,t\in [0,\tau(x))$ with $s\leq t$
\begin{equation*}
\frac{\theta_{f}(t,x)}{\theta_{f}(s,x)} \leq \frac{    e^{-\int^{t}_{0}\,F^{-2}_{x}(u)F_{\kappa,\lambda,x}(u)\,du}     }{    e^{-\int^{s}_{0}\,F^{-2}_{x}(u)F_{\kappa,\lambda,x}(u)\,du}     }.
\end{equation*}
Therefore,
we have the lemma.
\end{proof}
\subsection{Equality cases}
We recall the following radial curvature equation (see e.g., Theorem 2 in \cite{Pe}):
\begin{lem}\label{lem:radial curvature equation}
Let $\rho$ be a smooth function defined on a domain in $M$ such that $\Vert \nabla \rho \Vert=1$.
Let $X$ be a parallel vector field along an integral curve of $\nabla \rho$
that is orthogonal to $\nabla \rho$.
Then we have
\begin{equation*}\label{eq:radial curvature equation}
g(R(X,\nabla \rho)\nabla \rho,X)=g(\nabla_{\nabla \rho}  A_{\nabla \rho}X,X)-g(A_{\nabla \rho}A_{\nabla \rho}X,X),
\end{equation*}
where $R$ is the curvature tensor induced from $g$,
and $A_{\nabla \rho}$ is the shape operator of the level set of $\rho$ toward $\nabla \rho$.
In particular,
if there exists a function $\varphi$ defined on the domain of the integral curve such that $A_{\nabla \rho}X=-\varphi\,X$,
then we have
\begin{equation*}
g(R(X,\nabla \rho)\nabla \rho,X)=-(\varphi'+\varphi^{2})\Vert X \Vert^{2}.
\end{equation*}
\end{lem}
For the equality case of Lemma \ref{lem:Laplacian comparison},
we have:
\begin{lem}\label{lem:Equality in Laplacian comparison}
Take $x\in \bm$.
For $\kappa,\lambda \in \mathbb{R}$ and for $N\in (-\infty,1]$
we suppose that for all $t\in (0,\tau(x))$
we have $\ric^{N}_{f}(\gamma'_{x}(t)) \geq \kappa$,
and suppose $H_{f,x}\geq \lambda$.
Choose an orthonormal basis $\{e_{x,i}\}_{i=1}^{n-1}$ of $T_{x}\bm$,
and let $\{Y_{x,i}\}^{n-1}_{i=1}$ be the $\bm$-Jacobi fields along $\gamma_{x}$
with initial conditions $Y_{x,i}(0)=e_{x,i}$ and $Y_{x,i}'(0)=-A_{u_{x}}e_{x,i}$.
Assume that
for some $t_{0}\in (0,\tau(x))$
the equality in $(\ref{eq:Laplacian comparison})$ holds.
Then $\kappa=0$ and $\lambda=0$,
and for all $i$
we have $Y_{x,i}=F_{x}\,E_{x,i}$ on $[0,t_{0}]$,
where $\{E_{x,i}\}^{n-1}_{i=1}$ are the parallel vector fields along $\gamma_{x}$ with initial condition $E_{x,i}(0)=e_{x,i}$.
Moreover,
if $N \in (-\infty,1)$,
then $f \circ \gamma_{x}$ is constant on $[0,t_{0}]$;
in particular,
$Y_{x,i}=E_{x,i}$.
\end{lem}
\begin{proof}
By the equality assumption,
there exists a function $\varphi$ on the set $\gamma_{x}((0,t_{0}))$ such that
at each point on $\gamma_{x}((0,t_{0}))$
we have $\Hess \rho_{\bm}=\varphi\,g$ on the orthogonal complement of $\nabla \rho_{\bm}$ (see Remark \ref{rem:equality in the lemma}).
Put $\varphi_{x}:=\varphi \circ \gamma_{x}$.
For each $i$,
it holds that
\begin{equation*}
g(A_{\nabla \rho_{\bm}}E_{x,i},E_{x,i})=-\Hess \rho_{\bm}(E_{x,i},E_{x,i})=-\varphi_{x}.
\end{equation*}
It follows that $A_{\nabla \rho_{\bm}}E_{x,i}=-\varphi_{x}\,E_{x,i}$.
From Lemma \ref{lem:radial curvature equation},
we derive
\begin{equation}\label{eq:curvature}
R(E_{x,i},\nabla \rho_{\bm})\nabla \rho_{\bm}=-(\varphi'_{x}+\varphi^{2}_{x})E_{x,i}.
\end{equation}

Put $f_{x}:=f \circ \gamma_{x}$ and $h_{f,x}:=\left(\Delta_{f}\,\rho_{\bm}\right)\circ \gamma_{x}$.
By the equality assumption,
we have $F^{2}_{x}\,h_{f,x}=F_{\kappa,\lambda,x}$ on $[0,t_{0}]$ (see Remark \ref{rem:equality in the lemma}).
Fix $t\in [0,t_{0}]$.
Since $h_{f,x}(t)=F^{-2}_{x}(t)\,F_{\kappa,\lambda,x}(t)$,
we have 
\begin{equation*}
\Delta \rho_{\bm}(\gamma_{x}(t))=-f'_{x}(t)+F^{-2}_{x}(t)\,F_{\kappa,\lambda,x}(t).
\end{equation*}
On the other hand,
from $\Hess \rho_{\bm}=\varphi\,g$,
we deduce $\Delta \rho_{\bm}(\gamma_{x}(t))=-(n-1)\varphi_{x}(t)$.
Hence,
$\varphi_{x}(t)$ is equal to
\begin{equation*}
(n-1)^{-1}\left(f_{x}(t)-f(x)-\int^{t}_{0}\,F^{-2}_{x}(s)\,F_{\kappa,\lambda,x}(s)\,ds\right)'.
\end{equation*}
Now,
define a function $\mathcal{F}_{x}:[0,t_{0}] \to (0,\infty)$ by
\begin{equation*}
\mathcal{F}_{x}(t):=e^{\int^{t}_{0}\, \varphi_{x}(s)\,ds}
=e^{-\frac{\int^{t}_{0}\, F^{-2}_{x}(s)F_{\kappa,\lambda,x}(s)  \,ds}{n-1}}\,F_{x}(t).
\end{equation*}
Note that
if $\kappa=0$ and $\lambda=0$,
then $\mathcal{F}_{x}=F_{x}$.
By (\ref{eq:curvature}),
a vector field $\tilde{Y}_{x,i}$ along $\gamma_{x}|_{[0,t_{0}]}$ defined by $\tilde{Y}_{x,i}:=\mathcal{F}_{x}\,E_{x,i}$ is
a $\bm$-Jacobi field along $\gamma_{x}|_{[0,t_{0}]}$
with initial conditions $Y_{x,i}(0)=e_{x,i}$ and $Y_{x,i}'(0)=-A_{u_{x}}e_{x,i}$.
Therefore,
$Y_{x,i}$ coincides with $\tilde{Y}_{x,i}$ on $[0,t_{0}]$.

From (\ref{eq:curvature}) and
$Y_{x,i}=\mathcal{F}_{x}\,E_{x,i}$,
we derive $R(E_{x,i},\nabla \rho_{\bm})\nabla \rho_{\bm}=-\mathcal{F}''_{x}\,\mathcal{F}^{-1}_{x}\,E_{x,i}$.
This implies that
for each $t \in (0,t_{0})$
\begin{align*}
\ric_{g}(\gamma'_{x}(t))&=-(n-1)\mathcal{F}''_{x}(t)\,\mathcal{F}^{-1}_{x}(t)\\
                                      &=-f''_{x}(t)-\frac{f'_{x}(t)^{2}}{n-1}+\kappa-\frac{F^{-4}_{x}(t)\,F^{2}_{\kappa,\lambda,x}(t)}{n-1};
\end{align*}
in particular,
$\ric^{1}_{f}(\gamma'_{x}(t))=\kappa-(n-1)^{-1}F^{-4}_{x}(t)F^{2}_{\kappa,\lambda,x}(t)$.
By the monotonicity of $\ric^{N}_{f}$ with respect to $N$,
we see
\begin{equation*}
\kappa \leq \ric^{N}_{f}(\gamma'_{x}(t)) \leq \ric^{1}_{f}(\gamma'_{x}(t))=\kappa-\frac{F^{-4}_{x}(t)\,F^{2}_{\kappa,\lambda,x}(t)}{n-1},
\end{equation*}
and hence $F_{\kappa,\lambda,x}(t)=0$.
We obtain $\kappa=0$ and $\lambda=0$,
and $Y_{x,i}=F_{x}\,E_{x,i}$.

Now,
we have $(1-N)(f'_{x})^{2}=0$ on $[0,t_{0}]$ (see Remark \ref{rem:equality in the lemma}).
If $N$ is smaller than $1$,
then $f'_{x}=0$ on $[0,t_{0}]$;
in particular,
$f_{x}$ is constant on $[0,t_{0}]$.
This completes the proof.
\end{proof}
By Lemma \ref{lem:Equality in Laplacian comparison},
we have the following:
\begin{lem}\label{lem:Equality in Basic comparison}
Take $x\in \bm$.
Let $\kappa \in \mathbb{R}$ and $\lambda \in \mathbb{R}$ satisfy the subharmonic-condition.
For $N\in (-\infty,1]$
we suppose that for all $t\in (0,\tau(x))$
we have $\ric^{N}_{f}(\gamma'_{x}(t)) \geq \kappa$,
and suppose $H_{f,x}\geq \lambda$.
We choose an orthonormal basis $\{e_{x,i}\}_{i=1}^{n-1}$ of $T_{x}\bm$,
and let $\{Y_{x,i}\}^{n-1}_{i=1}$ be the $\bm$-Jacobi fields along $\gamma_{x}$
with initial conditions $Y_{x,i}(0)=e_{x,i}$ and $Y_{x,i}'(0)=-A_{u_{x}}e_{x,i}$.
Assume that
for some $t_{0} \in (0,\tau(x))$
we have $\Delta_{f}\rho_{\bm}(\gamma_{x}(t_{0}))=0$.
Then $\kappa=0$ and $\lambda=0$,
and for all $i$
we have $Y_{x,i}=F_{x}\,E_{x,i}$ on $[0,t_{0}]$,
where $\{E_{x,i}\}^{n-1}_{i=1}$ are the parallel vector fields along $\gamma_{x}$ with initial condition $E_{x,i}(0)=e_{x,i}$.
Moreover,
if $N \in (-\infty,1)$,
then $f \circ \gamma_{x}$ is constant on $[0,t_{0}]$;
in particular,
$Y_{x,i}=E_{x,i}$.
\end{lem}
\begin{proof}
The assumption $\Delta_{f}\rho_{\bm}(\gamma_{x}(t_{0}))=0$ implies that
the equality in $(\ref{eq:Laplacian comparison})$ holds.
Lemma \ref{lem:Equality in Laplacian comparison} leads to the lemma.
\end{proof}

\subsection{Distributions}\label{sec:Distributions}
From Lemma \ref{lem:Laplacian comparison},
we derive the following:
\begin{lem}\label{lem:p-Laplacian comparison}
Take $x\in \bm$.
Let $p\in (1,\infty)$,
and let $\kappa \in \mathbb{R}$ and $\lambda \in \mathbb{R}$ satisfy the subharmonic-condition.
For $N \in (-\infty,1]$
we suppose that for all $t\in (0,\tau(x))$
we have $\ric^{N}_{f}(\gamma'_{x}(t)) \geq \kappa$,
and suppose $H_{f,x}\geq \lambda$.
Let $\varphi:[0,\infty)\to \mathbb{R}$ be a monotone increasing smooth function.
Then for all $t\in (0,\tau(x))$
\begin{equation}\label{eq:p-Laplacian comparison}
\Delta_{f,p} (\varphi \circ \rho_{\bm}) (\gamma_{x}(t)) \geq -\left(     \left(  \varphi'  \right)^{p-1}     \right)'(t).
\end{equation}
\end{lem}
\begin{proof}
For all $t\in (0,\tau(x))$
we see
\begin{equation*}
\Delta_{f,p}\, (\varphi \circ \rho_{\bm}) (\gamma_{x}(t))=- \left(     \left(  \varphi'  \right)^{p-1}     \right)'(t)  +\Delta_{f,2}\, \rho_{\bm}(\gamma_{x}(t))\, \varphi'(t)^{p-1}.
\end{equation*}
This together with Lemma \ref{lem:Laplacian comparison} implies (\ref{eq:p-Laplacian comparison}).
\end{proof}
\begin{rem}\label{rem:Equality in p-Laplacian comparison}
The equality case in Lemma \ref{lem:p-Laplacian comparison}
corresponds to that in Lemma \ref{lem:Laplacian comparison}
(see Lemma \ref{lem:Equality in Basic comparison}).
\end{rem}
From Lemma \ref{lem:p-Laplacian comparison},
we deduce the following:
\begin{prop}\label{prop:global p-Laplacian comparison}
Let $p\in (1,\infty)$,
and let $\kappa \in \mathbb{R}$ and $\lambda \in \mathbb{R}$ satisfy the subharmonic-condition.
For $N\in (-\infty,1]$
we suppose $\ric^{N}_{f,M} \geq \kappa$,
and $H_{f,\bm}\geq \lambda$.
For a monotone increasing smooth function $\varphi:[0,\infty)\to \mathbb{R}$,
we put $\Phi:=\varphi \circ \rho_{\bm}$.
Then
\begin{equation*}
\Delta_{f,p}\,\Phi \geq  -\left( \left(\varphi' \right)^{p-1} \right)' \circ \rho_{\bm}
\end{equation*}
in a distribution sense on $M$.
More precisely,
for every non-negative smooth function $\psi:M\to \mathbb{R}$ whose support is compact and contained in $\inte M$,
we have
\begin{equation}\label{eq:global p-Laplacian comparison}
\int_{M}\,   \Vert \nabla \Phi \Vert^{p-2} g\left(\nabla \psi, \nabla \Phi  \right)\, d\,m_{f}\\
\geq \int_{M}\,\psi \left( -\left( \left(\varphi' \right)^{p-1} \right)' \circ \rho_{\bm}\right)   \,d\,m_{f}.
\end{equation}
\end{prop}
\begin{proof}
By Lemma \ref{lem:avoiding the cut locus},
there exists a sequence $\{\Omega_{k}\}_{k\in \mathbb{N}}$ of closed subsets of $M$ such that
for every $k$,
the set $\partial \Omega_{k}$ is a smooth hypersurface in $M$ satisfying the following:
(1) for all $k_{1},k_{2}\in \mathbb{N}$ with $k_{1}<k_{2}$,
we have $\Omega_{k_{1}}\subset \Omega_{k_{2}}$;
(2) $M\setminus \cut \bm=\bigcup_{k}\,\Omega_{k}$;
(3) $\partial \Omega_{k}\cap \bm=\bm$ for all $k$;
(4) for each $k$,
on $\partial \Omega_{k}\setminus \bm$,
there exists the unit outer normal vector field $\nu_{k}$ for $\Omega_{k}$ with $g(\nu_{k},\nabla \rho_{\bm})\geq 0$.
For the canonical Riemannian volume measure $\vol_{k}$ on $\partial \Omega_{k}\setminus \bm$,
put $m_{f,k}:=e^{-f|_{\partial \Omega_{k}\setminus \bm}}\,\vol_{k}$.
Let $\psi:M\to \mathbb{R}$ be a non-negative smooth function whose support is compact and contained in $\inte M$.
By the Green formula,
and by $\partial \Omega_{k}\cap \bm=\bm$,
\begin{multline*}
\int_{\Omega_{k}}\,  \Vert \nabla \Phi \Vert^{p-2} g\left(\nabla \psi, \nabla \Phi  \right)  \, d\,m_{f}\\
=\int_{\Omega_{k}}\, \psi\,\Delta_{f,p} \Phi \,d\,m_{f}+\int_{\partial \Omega_{k}\setminus \bm}\, \Vert \nabla \Phi \Vert^{p-2}\, \psi \,g\left(\nu_{k},\nabla \Phi \right) \, d\,m_{f,k}.
\end{multline*}
From Lemma \ref{lem:p-Laplacian comparison} and $g(\nu_{k},\nabla \rho_{\bm})\geq 0$,
we derive
\begin{equation*}\label{eq:weak global p-Laplacian comparison}
\int_{\Omega_{k}}\,   \Vert \nabla \Phi \Vert^{p-2} g\left(\nabla \psi, \nabla \Phi  \right)\, d\,m_{f}
\geq \int_{\Omega_{k}}\,\psi\,  \left( -\left( \left(\varphi' \right)^{p-1} \right)' \circ \rho_{\bm}\right)\,d\,m_{f}.
\end{equation*}
Letting $k \to \infty$,
we have the proposition.
\end{proof}
\begin{rem}\label{rem:Equality case in global p-Laplacian comparison}
In Proposition \ref{prop:global p-Laplacian comparison},
we assume that
the equality in (\ref{eq:global p-Laplacian comparison}) holds.
Then for every $x\in \bm$,
and for every $t\in (0,\tau(x))$,
the equality in (\ref{eq:p-Laplacian comparison}) also holds.
The equality case in Proposition \ref{prop:global p-Laplacian comparison}
corresponds to that in Lemma \ref{lem:p-Laplacian comparison} (see Remark \ref{rem:Equality in p-Laplacian comparison}).
\end{rem}
\begin{rem}
Perales \cite{P} has shown a Laplacian comparison inequality for the distance function from the boundary in a barrier sense
for manifolds with boundary of non-negative Ricci curvature.
We can prove that
the Laplacian comparison inequality for $\rho_{\bm}$ in Lemma \ref{lem:Laplacian comparison} globally holds on $M$ in a barrier sense.
\end{rem}

%% file: section4.tex
\section{Volume comparisons}\label{sec:Volume comparisons}
Let $M$ be an $n$-dimensional,
connected complete Riemannian manifold with boundary with Riemannian metric $g$,
and let $f:M\to \mathbb{R}$ be a smooth function.
\subsection{Absolute volume comparisons}\label{sec:Absolute volume comparison}
We have the following absolute volume comparison inequality of Heintze-Karcher type:
\begin{lem}\label{lem:absolute volume comparison}
Suppose that
$\bm$ is compact.
For $\kappa,\lambda \in \mathbb{R}$,
and for $N\in (-\infty,1]$
we suppose $\ric^{N}_{f,M}\geq \kappa$,
and $H_{f,\bm} \geq \lambda$.
Then for all $r\in (0,\infty)$
we have
\begin{equation*}
m_{f}(B_{r}(\bm))\leq \int_{\bm}\int^{\min\{r,\tau(x)\}}_{0}\,e^{-\int^{t}_{0}\,F^{-2}_{x}(u)\,F_{\kappa,\lambda,x}(u)\,du}\,dt\, dm_{f,\bm},
\end{equation*}
where $F_{\kappa,\lambda,x}$ is the function defined as $(\ref{eq:subharmonic warping function})$.
In particular,
if $\kappa$ and $\lambda$ satisfy the subharmonic-condition,
then for all $r\in (0,\infty)$
we have $m_{f}(B_{r}(\bm))\leq r\, m_{f,\bm}(\bm)$,
and hence $(\ref{eq:volume growth upper estimate})$.
\end{lem}
\begin{proof}
Define a function $\tilde{\theta}:[0,\infty) \times \bm \to \mathbb{R}$ by
\begin{equation*}\label{ref:model density}
\tilde{\theta}(t,x) := \begin{cases}
                            e^{-\int^{t}_{0}\,F^{-2}_{x}(u)\,F_{\kappa,\lambda,x}(u)\,du} & \text{if $t< \tau(x)$}, \\
                            0           & \text{if $t \geq \tau(x)$}.
                       \end{cases}
\end{equation*}
By Lemma \ref{lem:Basic comparison},
for all $x \in \bm$ and $t \in (0,\infty)$
we see $\bar{\theta}_{f}(t,x) \leq \tilde{\theta}(t,x)\,e^{-f(x)}$,
where $\bar{\theta}_{f}$ is the function defined as (\ref{eq:extend volume element}).
Integrate the both sides of the inequality over $(0,r)$ with respect to $t$,
and then do that over $\bm$ with respect to $x$.
From Lemma \ref{lem:Basic volume lemma},
we deduce
\begin{equation*}
m_{f}(B_{r}(\bm)) \leq \int_{\bm}\int^{r}_{0}\,\tilde{\theta}(t,x) \,dt\, dm_{f,\bm}.
\end{equation*}
This implies the lemma.
\end{proof}
\begin{rem}
Under a lower $N$-weighted Ricci curvature bound,
Bayle \cite{B} has stated an inequality of Heintze-Karcher type without proof in the case of $N \in [n,\infty)$.
Morgan \cite{Mo} has proved it in the case of $N=\infty$,
and Milman \cite{M} has done in the case of $N \in (-\infty,1)$.
\end{rem}

\subsection{Relative volume comparisons}\label{sec:Relative volume comparison}
We have the following relative volume comparison theorem of Bishop-Gromov type:
\begin{thm}\label{thm:relative volume comparison}
Let $M$ be a connected complete Riemannian manifold with boundary,
and let $f:M\to \mathbb{R}$ be a smooth function.
Suppose that
$\bm$ is compact.
Let $\kappa \in \mathbb{R}$ and $\lambda \in \mathbb{R}$ satisfy the subharmonic-condition.
For $N\in (-\infty,1]$
we suppose $\ric^{N}_{f,M}\geq \kappa$,
and $H_{f,\bm}\geq \lambda$.
Then for all $r,R\in (0,\infty)$ with $r\leq R$
\begin{equation}\label{eq:relative volume comparison}
\frac{m_{f} (B_{R}(\bm))}{m_{f}(B_{r}(\bm))}\leq \frac{R}{r}.
\end{equation}
\end{thm}
\begin{proof}
Lemma \ref{lem:Basic comparison} implies that
for all $s,t\in [0,\infty)$ with $s\leq t$
we have $\bar{\theta}_{f}(t,x) \leq \bar{\theta}_{f}(s,x)$,
where $\bar{\theta}_{f}$ is the function defined as (\ref{eq:extend volume element}).
By integrating the both sides over $(0,r)$ with respect to $s$,
and then doing that over $(r,R)$ with respect to $t$,
we see
\begin{equation*}\label{eq:volume comparison3}
r\,\int^{R}_{r}\bar{\theta}_{f}(t,x)\,dt \leq (R-r)\,\int^{r}_{0}\bar{\theta}_{f}(s,x)\,ds.
\end{equation*}
From Lemma \ref{lem:Basic volume lemma},
we derive
\begin{equation*}
        \frac{m_{f}( B_{R}(\bm))}{m_{f}(B_{r}(\bm))}
 =    1+\frac{\int_{\bm} \int^{R}_{r}\bar{\theta}_{f}(t,x)\,dt\,d\vol_{h}}{\int_{\bm} \int^{r}_{0}\bar{\theta}_{f}(s,x)\,ds\,d\vol_{h}}
\leq 1+\frac{R-r}{r}
  =      \frac{R}{r}.
\end{equation*}
This proves the theorem.
\end{proof}
When $\kappa=0$ and $\lambda=0$,
Theorem \ref{thm:relative volume comparison} has been proved in the unweighted case in \cite{Sa1},
and in the usual weighted case in \cite{Sa2}.
\begin{rem}
In \cite{Sa1},
the author has proved a measure contraction inequality around the boundary in the unweighted case.
We can prove a similar measure contraction inequality in our setting.
The measure contraction inequality enables us to give another proof of Theorem \ref{thm:relative volume comparison}.
\end{rem}

\subsection{Volume growth rigidity}\label{sec:Volume growth rigidity}
We show the following:
\begin{lem}\label{lem:cut point lower bound}
Suppose that
$\bm$ is compact.
Let $\kappa \in \mathbb{R}$ and $\lambda \in \mathbb{R}$ satisfy the subharmonic-condition.
For $N \in (-\infty,1]$
we suppose $\ric^{N}_{f,M}\geq \kappa$,
and $H_{f,\bm} \geq \lambda$.
Assume that
there exists $R\in (0,\infty)$ such that
for every $r\in (0,R]$
the equality in $(\ref{eq:relative volume comparison})$ holds.
Then $\tau \geq R$ on $\bm$.
\end{lem}
\begin{proof}
The proof will be done by contradiction.
Suppose that
there exists $x_{0}\in \bm$
such that $\tau(x_{0})<R$.
Put $t_{0}:=\tau(x_{0})$.
Take $\epsilon \in (0,\infty)$ with $t_{0}+\epsilon<R$.
By the continuity of $\tau$,
there exists a closed geodesic ball $B$ in $\bm$ centered at $x_{0}$ such that
$\tau$ is smaller than or equal to $t_{0}+\epsilon$ on $B$.
Using Lemma \ref{lem:Basic comparison},
we see
\begin{equation*}
m_{f}(B_{R}(\bm))\leq R\,m_{f,\bm}(\partial M \setminus B)+(t_{0}+\epsilon)\,m_{f,\bm}(B)<R\,m_{f,\bm}(\bm).
\end{equation*}
On the other hand,
$m_{f}(B_{R}(\bm))/m_{f,\bm}(\bm)$ is equal to $R$.
This is a contradiction.
\end{proof}
Suppose that
$\bm$ is compact.
Let $\kappa \in \mathbb{R}$ and $\lambda \in \mathbb{R}$ satisfy the subharmonic-condition.
For $N \in (-\infty,1]$
we suppose $\ric^{N}_{f,M}\geq \kappa$,
and $H_{f,\bm} \geq \lambda$.
Assume that
there exists $R \in (0,\infty)$ such that
for every $r\in (0,R]$
the equality in (\ref{eq:relative volume comparison}) holds.
Then for each $r\in (0,R)$
the level set $\rho_{\bm}^{-1}(r)$ is an $(n-1)$-dimensional submanifold of $M$ (see Lemma \ref{lem:cut point lower bound}).
In particular,
$(B_{r}(\bm),g)$ is an $n$-dimensional (not necessarily, connected) complete Riemannian manifold with boundary.
We denote by $d_{B_{r}(\bm)}$ and by $d_{[0,r]\times_{F} \bm}$
the Riemannian distances on $(B_{r}(\bm),g)$ and on $[0,r]\times_{F} \bm$,
respectively,
where $[0,r]\times_{F} \bm$ is the twisted product space $([0,r] \times \bm,dt^{2}+F_{x}^{2}(t)\,h)$.
\begin{lem}\label{lem:half rigidity}
Suppose that
$\bm$ is compact.
Let $\kappa \in \mathbb{R}$ and $\lambda \in \mathbb{R}$ satisfy the subharmonic-condition.
For $N \in (-\infty,1]$
we suppose $\ric^{N}_{f,M}\geq \kappa$,
and $H_{f,\bm} \geq \lambda$.
Assume that
there exists $R\in (0,\infty)$ such that
for every $r\in (0,R]$
the equality in $(\ref{eq:relative volume comparison})$ holds.
Then for each $r\in (0,R)$
the metric space $(B_{r}(\bm),d_{B_{r}(\bm)})$ is isometric to $([0,r]\times_{F} \bm,d_{[0,r]\times_{F} \bm})$.
Moreover,
if $N \in (-\infty,1)$,
then for every $x\in \bm$
the function $f\circ \gamma_{x}$ is constant on $[0,r]$;
in particular,
$(B_{r}(\bm),d_{B_{r}(\bm)})$ is isometric to $([0,r]\times \bm,d_{[0,r]\times \bm})$.
\end{lem}
\begin{proof}
Since each connected component of $\bm$ one-to-one corresponds to the connected component of $B_{r}(\bm)$,
it suffices to consider the case where
$B_{r}(\bm)$ is connected.
For each $x\in \bm$
we choose an orthonormal basis $\{ e_{x,i} \}_{i=1}^{n-1}$ of $T_{x}\bm$.
Let $\{Y_{x,i}\}^{n-1}_{i=1}$ be the $\bm$-Jacobi fields along $\gamma_{x}$
with initial conditions $Y_{x,i}(0)=e_{x,i}$ and $Y_{x,i}'(0)=-A_{u_{x}}e_{x,i}$.
Since the equality in $(\ref{eq:relative volume comparison})$ holds,
for all $t \in [0,r]$
we see $\theta_{f}(t,x)=\theta_{f}(r,x)$.
By (\ref{eq:Laplacian representation}),
for all $t\in (0,r]$
we see $\Delta_{f}\rho_{\bm}(\gamma_{x}(t))=0$.
Lemma \ref{lem:Equality in Basic comparison} implies that
we have $\kappa=0$ and $\lambda=0$,
and for all $i$
we have $Y_{x,i}=F_{x}E_{x,i}$ on $[0,r]$,
where $\{E_{x,i}\}^{n-1}_{i=1}$ are the parallel vectors field along $\gamma_{x}$ with initial condition $E_{x,i}(0)=e_{x,i}$;
moreover,
if $N\in (-\infty,1)$,
then $f \circ \gamma_{x}$ is constant on $[0,r]$.
Define a map $\Phi:[0,r]\times \bm\to B_{r}(\bm)$ by $\Phi(t,x):=\gamma_{x}(t)$.
For each $p\in (0,r)\times \bm$
the differential map $D(\Phi|_{(0,r)\times \bm})_{p}$ sends an orthonormal basis of $T_{p}([0,r]\times \bm)$ to that of $T_{\Phi(p)}B_{r}(\bm)$,
and for each $x\in \{0,r\}\times \bm$
the map $D(\Phi|_{\{0,r\}\times \bm})_{x}$ sends an orthonormal basis of $T_{x}(\{0,r\}\times \bm)$ to that of $T_{\Phi(x)}\partial (B_{r}(\bm))$.
Hence,
$\Phi$ is a Riemannian isometry with boundary from $[0,r]\times_{F} \bm$ to $B_{r}(\bm)$.
\end{proof}
Now,
we are in a position to prove Theorem \ref{thm:volume growth rigidity} and Corollary \ref{cor:warped volume growth rigidity}.
\begin{proof}[Proof of Theorem \ref{thm:volume growth rigidity}]
Suppose that
$\bm$ is compact.
Let $\kappa \in \mathbb{R}$ and $\lambda \in \mathbb{R}$ satisfy the subharmonic-condition.
For $N \in (-\infty,1]$
we suppose $\ric^{N}_{f,M}\geq \kappa$,
and $H_{f,\bm} \geq \lambda$.
Furthermore,
we assume (\ref{eq:volume growth rigidity}).
By Lemma \ref{lem:absolute volume comparison} and Theorem \ref{thm:relative volume comparison},
for every $R\in (0,\infty)$,
and for every $r \in (0,R]$,
\begin{equation*}
\frac{m_{f}(B_{R}(\bm))}{R}=\frac{m_{f}(B_{r}(\bm))}{r}=m_{f,\bm}(\bm);
\end{equation*}
in particular,
the equality in $(\ref{eq:relative volume comparison})$ holds.
From Lemma \ref{lem:cut point lower bound},
we deduce $\tau=\infty$ on $\bm$.
We see $\cut \bm=\emptyset$,
and hence $\bm$ is connected.
Take a sequence $\{r_{i}\}$ with $r_{i}\to \infty$.
By Lemma \ref{lem:half rigidity},
for every $i$
there exists a Riemannian isometry $\Phi_{i}:[0,r_{i}]\times \bm\to B_{r_{i}}(\bm)$ with boundary
from $[0,r_{i}]\times_{F} \bm$ to $B_{r_{i}}(\bm)$ defined by $\Phi_{i}(t,x):=\gamma_{x}(t)$.
Moreover,
if $N\in (-\infty,1)$,
then for each $x\in \bm$
the function $f \circ \gamma_{x}$ is constant on $[0,r_{i}]$.
Since $\cut \bm=\emptyset$,
we obtain a Riemannian isometry $\Phi:[0,\infty)\times\bm\to M$ with boundary
from $[0,\infty)\times_{F} \bm$ to $M$
defined by $\Phi(t,x):=\gamma_{x}(t)$ such that $\Phi|_{[0,r_{i}]\times \bm}=\Phi_{i}$ for all $i$.
Furthermore,
if $N\in (-\infty,1)$,
then $f \circ \gamma_{x}$ is constant on $[0,\infty)$.
Thus,
we complete the proof of Theorem \ref{thm:volume growth rigidity}.
\end{proof}
\medskip

\begin{proof}[Proof of Corollary \ref{cor:warped volume growth rigidity}]
Combining Theorem \ref{thm:volume growth rigidity} and Proposition \ref{prop:twisted to warped},
we conclude Corollary \ref{cor:warped volume growth rigidity}.
\end{proof}

%% file: section5.tex
\section{Splitting theorems}\label{sec:Splitting theorems}
Let $M$ be an $n$-dimensional, 
connected complete Riemannian manifold with boundary with Riemannian metric $g$,
and let $f:M\to \mathbb{R}$ be a smooth function.
\subsection{Basic splitting}
Let $\varphi:M\to \mathbb{R}$ be a continuous function,
and let $U$ be a domain contained in $\inte M$.
For $p\in U$,
and for a function $\psi$ defined on an open neighborhood of $p$,
we say that
$\psi$ is a \textit{support function of $\varphi$ at $p$}
if we have $\psi(p)=\varphi(p)$ and $\psi \leq \varphi$.
We say that
$\varphi$ is \textit{$f$-subharmonic on $U$}
if for every $p\in U$,
and for every $\epsilon \in (0,\infty)$,
there exists a smooth,
support function $\psi_{p,\epsilon}$ of $\varphi$ at $p$
such that $\Delta_{f}\, \psi_{p,\epsilon}(p)\leq \epsilon$.

We recall the following maximal principle of Calabi type (see e.g., \cite{C}, and Lemma 2.4 in \cite{FLZ}).
\begin{lem}\label{lem:maximal principle}
Let $U$ be a domain contained in $\inte M$.
If a $f$-subharmonic function on $U$ takes the maximal value at a point in $U$,
then it must be constant on $U$.
\end{lem}

Wylie \cite{W} has proved a subharmonicity of Busemann functions on manifolds without boundary (see Lemma 3.4 in \cite{W}).
In our case,
under an assumption concerning asymptotes for a ray defined in Subsection \ref{subsec:Busemann functions and asymptotes},
the subharmonicity holds in the following form:
\begin{lem}[\cite{W}]\label{lem:subharmonicity of busemann function}
Assume $\sup f(M)<\infty$.
For $N \in (-\infty,1]$
we suppose $\ric^{N}_{f,M}\geq 0$.
Let $\gamma:[0,\infty)\to M$ be a ray that lies in $\inte M$,
and let $U$ be a domain contained in $\inte M$ such that
for each $p\in U$,
there exists an asymptote for $\gamma$ from $p$
that lies in $\inte M$.
Then the Busemann function $b_{\gamma}$ of $\gamma$ is $f$-subharmonic on $U$. 
\end{lem}
Now,
we prove Theorem \ref{thm:splitting theorem} and Corollary \ref{cor:warped splitting theorem}.
\begin{proof}[Proof of Theorem \ref{thm:splitting theorem}]
Assume $\sup f(M)<\infty$.
For $N\in (-\infty,1]$
we suppose $\ric^{N}_{f,M}\geq 0$,
and $H_{f,\bm} \geq 0$.
Suppose that
for some $x_{0} \in \bm$ we have $\tau(x_{0})=\infty$.

For the connected component $\bm_{0}$ of $\bm$ containing $x_{0}$,
put
\begin{equation*}
\Omega:=\{y\in \bm_{0} \mid \tau(y)=\infty \}.
\end{equation*}
By the continuity of $\tau$,
the set $\Omega$ is a non-empty closed subset of $\bm_{0}$.

We show the openness of $\Omega$ in $\bm_{0}$.
Fix $y_{0}\in \Omega$.
Take $l \in (0,\infty)$,
and put $p_{0}:=\gamma_{y_{0}}(l)$.
There exists an open neighborhood $U$ of $p_{0}$ in $\inte M$ contained in $D_{\bm}$.
Taking $U$ smaller,
we may assume that 
for every $q\in U$
the unique foot point on $\bm$ of $q$ belongs to $\bm_{0}$.
By Lemma \ref{lem:asymptote},
there exists $\epsilon \in (0,\infty)$ such that
for all $q\in B_{\epsilon}(p_{0})$,
all asymptotes for $\gamma_{y_{0}}$ from $q$ lie in $\inte M$.
We may assume $U\subset B_{\epsilon}(p_{0})$.
By Lemma \ref{lem:subharmonicity of busemann function},
$b_{\gamma_{y_{0}}}$ is $f$-subharmonic on $U$,
and by Lemma \ref{lem:Laplacian comparison},
$\Delta_{f} \rho_{\bm}\geq 0$ on $U$.
Hence,
$b_{\gamma_{y_{0}}}-\rho_{\bm}$ is $f$-subharmonic on $U$.
The function $b_{\gamma_{y_{0}}}-\rho_{\bm}$ takes the maximal value $0$ at $p_{0}$.
Lemma \ref{lem:maximal principle} implies that
$b_{\gamma_{y_{0}}}=\rho_{\bm}$ on $U$.
From Lemma \ref{lem:busemann function},
it follows that $\Omega$ is open in $\bm_{0}$.

Since $\bm_{0}$ is a connected component of $\bm$,
we have $\Omega=\bm_{0}$.
By Lemma \ref{lem:splitting lemma},
$\bm$ is connected
and $\cut \bm=\emptyset$.
The equality in Lemma \ref{lem:Laplacian comparison} holds on $\inte M$.
For each $x\in \bm$
we choose an orthonormal basis $\{e_{x,i}\}_{i=1}^{n-1}$ of $T_{x}\bm$.
Let $\{Y_{x,i}\}^{n-1}_{i=1}$ be the $\bm$-Jacobi fields along $\gamma_{x}$
with initial conditions $Y_{x,i}(0)=e_{x,i}$ and $Y'_{x,i}(0)=-A_{u_{x}}e_{x,i}$.
By Lemma \ref{lem:Equality in Laplacian comparison},
for all $i$
we have $Y_{x,i}=F_{x}E_{x,i}$ on $[0,\infty)$,
where $\{E_{x,i}\}$ are the parallel vector fields along $\gamma_{x}$ with initial condition $E_{x,i}(0)=e_{x,i}$.
Moreover,
if $N\in (-\infty,1)$,
then $f \circ \gamma_{x}$ is constant on $[0,\infty)$.
Define a map $\Phi:[0,\infty)\times \bm\to M$ by $\Phi(t,x):=\gamma_{x}(t)$.
For every $p\in (0,\infty)\times \bm$
the differential map $D(\Phi|_{(0,\infty)\times \bm})_{p}$ sends an orthonormal basis of $T_{p}((0,\infty)\times \bm)$ to that of $T_{\Phi(p)}M$,
and for every $x\in \{0\}\times \bm$
the map $D(\Phi|_{\{0\}\times \bm})_{x}$ sends an orthonormal basis of $T_{x}(\{0\}\times \bm)$ to that of $T_{\Phi(x)}\bm$.
Then
$\Phi$ is a Riemannian isometry with boundary from $[0,\infty)\times_{F} \bm$ to $M$.
This proves Theorem \ref{thm:splitting theorem}.
\end{proof}
\medskip

\begin{proof}[Proof of Corollary \ref{cor:warped splitting theorem}]
From Theorem \ref{thm:splitting theorem} and Proposition \ref{prop:twisted to warped},
we derive Corollary \ref{cor:warped splitting theorem}.
\end{proof}
Lemma \ref{lem:bmcompact} and the continuity of $\tau$ imply that
if $M$ is non-compact and $\bm$ is compact,
then for some $x\in \bm$ we have $\tau(x)=\infty$.
We have the following corollary of Theorem \ref{thm:splitting theorem}.
\begin{cor}\label{cor:Kasue splitting}
Let $M$ be a connected complete Riemannian manifold with boundary,
and let $f:M\to \mathbb{R}$ be a smooth function such that $\sup f(M)<\infty$.
For $N\in (-\infty,1]$
we suppose $\ric^{N}_{f,M}\geq 0$,
and $H_{f,\bm} \geq 0$.
If $M$ is non-compact and $\bm$ is compact,
then $(M,d_{M})$ is isometric to $([0,\infty)\times_{F}\bm,d_{[0,\infty)\times_{F}\bm})$.
Moreover,
if $N\in (-\infty,1)$, 
then for every $x\in \bm$
the function $f\circ \gamma_{x}$ is constant on $[0,\infty)$;
in particular,
$(M,d_{M})$ is isometric to $([0,\infty)\times \bm,d_{[0,\infty)\times \bm})$.
\end{cor}
\subsection{Weighted Ricci curvature on the boundary}\label{sec:Weighted Ricci curvature on the boundary}
For $x \in \bm$,
we recall that
$u_{x}$ denotes the unit inner normal vector on $\bm$ at $x$.

The following seems to be well-known,
especially in a submanifold setting (see e.g., Proposition 9.36 in \cite{Be}, and Lemma 5.4 in \cite{Sa1}).
\begin{lem}\label{lem:boundary Ricci curvature}
Take $x\in \bm$,
and a unit vector $u$ in $T_{x}\bm$.
Choose an orthonormal basis $\{ e_{x,i} \}_{i=1}^{n-1}$ of $T_{x}\bm$ with $e_{x,1}=u$.
Then we have
\begin{equation*}
\ric_{h}(u)=\ric_{g}(u)-K_{g}(u_{x},u)+\tr A_{S(u,u)}-\sum_{i=1}^{n-1} \Vert S(u,e_{x,i})\Vert^{2},
\end{equation*}
where $h$ is the induced Riemannian metric on $\bm$,
and $K_{g}(u_{x},u)$ is the sectional curvature at $x$ in $(M,g)$ determined by $u_{x}$ and $u$.
\end{lem}
By using Lemma \ref{lem:boundary Ricci curvature},
we have the following:
\begin{lem}\label{lem:boundary weighted Ricci curvature}
Take $x\in \bm$,
and a unit vector $u$ in $T_{x}\bm$.
Choose an orthonormal basis $\{ e_{x,i} \}_{i=1}^{n-1}$ of $T_{x}\bm$ with $e_{x,1}=u$.
Then for all $N \in (-\infty,\infty)$
we have
\begin{align}\label{eq:boundary weighted Ricci curvature}
\ric^{N-1}_{f|_{\bm}}(u) & = \ric^{N}_{f}(u)+g((\nabla f)_{x},u_{x})\,g(S(u,u),u_{x})\\
                                    & - K_{g}(u_{x},u)+\tr A_{S(u,u)}-\sum_{i=1}^{n-1} \Vert S(u,e_{x,i})\Vert^{2},\notag
\end{align}
where $K_{g}(u_{x},u)$ is the sectional curvature at $x$ in $(M,g)$ determined by $u_{x}$ and $u$.
\end{lem}
\begin{proof}
First,
we assume $N\neq n$.
We see
\begin{align*}
h((\nabla(f|_{\bm}))_{x},u)&=g((\nabla f)_{x},u),\\ 
\Hess (f|_{\bm})(u,u)&=\Hess f(u,u)+g\left((\nabla f)_{x},u_{x}\right)\,g\left(S(u,u),u_{x}\right),
\end{align*}
where $h$ is the induced Riemannian metric on $\bm$.
Hence,
we have
\begin{multline*}
\ric^{N-1}_{f|_{\bm}}(u)=\ric_{h}(u)+\Hess (f|_{\bm})(u,u)-\frac{h((\nabla(f|_{\bm}))_{x},u)^{2}}{(N-1)-(n-1)}\\
                                    =\ric_{h}(u)+\Hess f(u,u)+g((\nabla f)_{x},u_{x})\,g(S(u,u),u_{x})-\frac{g((\nabla f)_{x},u)^{2}}{N-n}.
\end{multline*}
From Lemma \ref{lem:boundary Ricci curvature},
we derive (\ref{eq:boundary weighted Ricci curvature}).

Next,
we assume $N=n$.
If $f$ is constant,
then $\ric^{N-1}_{f|_{\bm}}(u)=\ric_{h}(u)$ and $\ric^{N}_{f}(u)=\ric_{g}(u)$,
and hence Lemma \ref{lem:boundary Ricci curvature} implies (\ref{eq:boundary weighted Ricci curvature}).
If $f$ is not constant,
then both the left hand side of (\ref{eq:boundary weighted Ricci curvature}) and the right hand side are equal to $-\infty$.
Therefore,
we complete the proof.
\end{proof}
From Lemma \ref{lem:boundary weighted Ricci curvature},
we derive the following:
\begin{lem}\label{lem:warped product boundary weighted Ricci curvature}
Take $x\in \bm$,
and a unit vector $u$ in $T_{x}\bm$.
If $(M,d_{M})$ is isometric to $([0,\infty)\times_{F}\bm,d_{[0,\infty)\times_{F}\bm})$,
then for all $N\in (-\infty,\infty)$
\begin{equation*}\label{eq:warped product boundary weighted Ricci curvature}
\ric^{N-1}_{f|_{\bm}}(u)=\ric^{N}_{f}(u)+\frac{\Hess f(u_{x},u_{x})}{n-1}.
\end{equation*}
\end{lem}
\begin{proof}
Choose an orthonormal basis $\{e_{x,i}\}_{i=1}^{n-1}$ of $T_{x}\bm$ with $e_{x,1}=u$.
Let $\{Y_{x,i}\}^{n-1}_{i=1}$ be the $\bm$-Jacobi fields along $\gamma_{x}$
with initial conditions $Y_{x,i}(0)=e_{x,i}$ and $Y'_{x,i}(0)=-A_{u_{x}}e_{x,i}$.
Since $(M,d_{M})$ is isometric to $([0,\infty)\times_{F}\bm,d_{[0,\infty)\times_{F}\bm})$,
there exists a Riemannian isometry with boundary from $M$ to $[0,\infty)\times_{F} \bm$.
In particular,
for all $i$
we see $Y_{x,i}=F_{x}\,E_{x,i}$,
where $\{E_{x,i}\}^{n-1}_{i=1}$ are the parallel vector fields along $\gamma_{x}$ with initial condition $E_{x,i}(0)=e_{x,i}$.
Hence,
for all $i$
\begin{equation}\label{eq:warped product shape operator}
A_{u_{x}}e_{x,i}=-Y'_{x,i}(0)=-\frac{g((\nabla f)_{x},u_{x})}{n-1}e_{x,i}.
\end{equation}
By (\ref{eq:warped product shape operator}),
for all $i \neq 1$
we have $S(u,e_{x,i})=0_{x}$,
and we have
\begin{equation}\label{eq:warped product second fundamental form}
S(u,u)=-\frac{g((\nabla f)_{x},u_{x})}{n-1}u_{x},\quad \tr A_{S(u,u)}=\frac{g((\nabla f)_{x},u_{x})^{2}}{n-1}.
\end{equation}
Furthermore,
the sectional curvature $K_{g}(u_{x},u)$ at $x$ in $(M,g)$ determined by $u_{x}$ and $u$ satisfies
\begin{equation}\label{eq:warped product sectional curvature}
K_{g}(u_{x},u)=-g(Y''_{x,1}(0),u)=-\left(\frac{\Hess f(u_{x},u_{x})}{n-1}+\left( \frac{g((\nabla f)_{x},u_{x})}{n-1}  \right)^{2}   \right).
\end{equation}
By Lemma \ref{lem:boundary weighted Ricci curvature},
and by (\ref{eq:warped product second fundamental form}) and (\ref{eq:warped product sectional curvature}),
we see
\begin{align*}
\ric^{N-1}_{f|_{\bm}}(u) &= \ric^{N}_{f}(u)-\frac{g((\nabla f)_{x},u_{x})^{2}}{n-1}+\frac{\Hess f(u_{x},u_{x})}{n-1}\\
                                     &+\left( \frac{g((\nabla f)_{x},u_{x})}{n-1}  \right)^{2}+\frac{g((\nabla f)_{x},u_{x})^{2}}{n-1}-\left( \frac{g((\nabla f)_{x},u_{x})}{n-1}  \right)^{2}\\
                                     &=\ric^{N}_{f}(u)+\frac{\Hess f(u_{x},u_{x})}{n-1}.
\end{align*}
This completes the proof.
\end{proof}
\subsection{Multi-splitting}\label{sec:Multi-splitting}
Let $M_{0}$ be a connected complete Riemannian manifold (without boundary).
A normal geodesic $\gamma:\mathbb{R} \to M_{0}$ is said to be a \textit{line}
if for all $s,t\in \mathbb{R}$ 
we have $d_{M_{0}}(\gamma(s),\gamma(t))=\vert s-t\vert$.

Wylie \cite{W} has proved the following splitting theorem of Cheeger-Gromoll type (see Theorem 1.2 and Corollary 1.3 in \cite{W}):
\begin{thm}[\cite{W}]\label{thm:splitting theorem of Cheeger-Gromoll type}
Let $M_{0}$ be a connected complete Riemannian manifold,
and let $f:M_{0}\to \mathbb{R}$ be a smooth function bounded from above.
For $N\in (-\infty,1]$
we suppose $\ric^{N}_{f,M_{0}}\geq 0$.
If $M_{0}$ contains a line,
then there exists a connected complete Riemannian manifold $N_{0}$ such that
$M_{0}$ is isometric to a warped product space over $\mathbb{R}\times N_{0}$.
Moreover,
if $N\in (-\infty,1)$,
then $M_{0}$ is isometric to the standard product $\mathbb{R}\times N_{0}$.
\end{thm}
\begin{rem}\label{rem:splitting theorem of Cheeger-Gromoll type}
For manifolds of non-negative $N$-weighted Ricci curvature,
Lichnerowicz \cite{Lic} has generalized the Cheeger-Gromoll splitting theorem in the case where $N=\infty$ and $f$ is bounded.
Fang, Li and Zhang \cite{FLZ} have done in the case where $N \in [n,\infty)$,
and in the case where $N=\infty$ and $f$ is bounded above.
\end{rem}
We obtain the following corollary of Theorem \ref{thm:splitting theorem}:
\begin{cor}\label{cor:boundary splitting}
Let $M$ be an $n$-dimensional,
connected complete Riemannian manifold with boundary,
and let $f:M\to \mathbb{R}$ be a smooth function such that $\sup f(M)<\infty$.
For $N\in (-\infty,1)$
we suppose $\ric^{N}_{f,M}\geq 0$,
and $H_{f,\bm} \geq 0$.
Assume that
for some $x_{0}\in \bm$
we have $\tau(x_{0})=\infty$.
Then there exist $k\in \{0,\dots,n-1\}$ and
an $(n-1-k)$-dimensional,
connected complete Riemannian manifold $N_{0}$ containing no line such that
$(\bm,d_{\bm})$ is isometric to $(\mathbb{R}^{k}\times N_{0},d_{\mathbb{R}^{k}\times N_{0}})$.
In particular,
$(M,d_{M})$ is isometric to $([0,\infty)\times \mathbb{R}^{k}\times N_{0},d_{[0,\infty)\times \mathbb{R}^{k}\times N_{0}})$.
\end{cor}
\begin{proof}
Due to Theorem \ref{thm:splitting theorem},
the metric space $(M,d_{M})$ is isometric to $([0,\infty)\times_{F} \bm,d_{[0,\infty)\times_{F} \bm})$.
Moreover,
for every $x\in \bm$,
the function $f\circ \gamma_{x}$ is constant on $[0,\infty)$;
in particular,
$(M,d_{M})$ is isometric to $([0,\infty)\times \bm,d_{[0,\infty)\times \bm})$.
For every $x\in \bm$,
it holds that $\Hess f(u_{x},u_{x})=0$.
By Lemma \ref{lem:warped product boundary weighted Ricci curvature},
we have $\ric^{N-1}_{f|_{\bm}}=\ric^{N}_{f}$ on the unit tangent bundle over $\bm$.
It follows that $\ric^{N-1}_{f|_{\bm},\bm} \geq 0$.
Note that
$N-1$ is smaller than $1$,
and $\sup_{x\in \bm}f(x)$ is finite.
Applying Theorem \ref{thm:splitting theorem of Cheeger-Gromoll type} to $\bm$ inductively,
we complete the proof.
\end{proof}
\subsection{Variants of the splitting theorem}\label{sec:Variants of splitting theorems}
We have already known several rigidity results studied in \cite{K2} (and \cite{CK}, \cite{I})
for manifolds with boundary whose boundaries are disconnected.
In \cite{Sa2},
the author has given generalizations of them in the usual weighted case (see Subsection 6.4 in \cite{Sa2}). 
We generalize one of them in our setting. 

For $A_{1}, A_{2}\subset M$,
we put $d_{M}(A_{1},A_{2}):=\inf_{p_{1}\in A_{1}, p_{2}\in A_{2}}\,d_{M}(p_{1},p_{2})$.

The following has been shown in \cite{K2} (see Lemma 1.6 in \cite{K2}):
\begin{lem}[\cite{K2}]\label{lem:disconnected lemma}
Suppose that
$\bm$ is disconnected.
Let $\{\bm_{i}\}_{i=1,2,\dots}$ denote the connected components of $\bm$.
Assume that $\bm_{1}$ is compact.
Put $D:=\inf_{i=2,3,\dots}\, d_{M}(\bm_{1},\bm_{i})$.
Then there exists a connected component $\bm_{2}$ of $\bm$ such that $d_{M}(\bm_{1},\bm_{2})=D$.
Furthermore,
for every $i=1,2$
there exists $x_{i}\in \bm_{i}$ such that $d_{M}(x_{1},x_{2})=D$.
The normal minimal geodesic $\gamma:[0,D]\to M$ from $x_{1}$ to $x_{2}$ is orthogonal to $\bm$ both at $x_{1}$ and at $x_{2}$,
and the restriction $\gamma|_{(0,D)}$ lies in $\inte M$.
\end{lem}

We prove the following splitting theorem:
\begin{thm}\label{thm:disconnected splitting}
Let $M$ be a connected complete Riemannian manifold with boundary,
and let $f:M\to \mathbb{R}$ be a smooth function.
Suppose that
$\bm$ is disconnected.
Let $\{\bm_{i}\}_{i=1,2,\dots}$ denote the connected components of $\bm$.
Suppose that
$\bm_{1}$ is compact,
and put $D:=\inf_{i=2,3,\dots}\, d_{M}(\bm_{1},\bm_{i})$.
Let $\kappa \in \mathbb{R}$ and $\lambda \in \mathbb{R}$ satisfy the subharmonic-condition.
For $N \in (-\infty,1]$
we suppose $\ric^{N}_{f,M}\geq \kappa$,
and $H_{f,\bm} \geq \lambda$.
Then $(M,d_{M})$ is isometric to $([0,D]\times_{F}\bm_{1},d_{[0,D]\times_{F} \bm_{1}})$.
Moreover,
if $N\in (-\infty,1)$,
then for every $x\in \bm_{1}$
the function $f\circ \gamma_{x}$ is constant on $[0,D]$;
in particular,
$(M,d_{M})$ is isometric to $([0,D] \times \bm_{1},d_{[0,D] \times \bm_{1}})$.
\end{thm}
\begin{proof}
By Lemma \ref{lem:disconnected lemma},
there exists a connected component $\bm_{2}$ of $\bm$ such that
$d_{M}(\bm_{1},\bm_{2})=D$.
For each $i=1,2$,
let $\rho_{\bm_{i}}:M\to \mathbb{R}$ be the distance function from $\bm_{i}$ defined as $\rho_{\bm_{i}}(p):=d_{M}(p,\bm_{i})$.
Put
\begin{equation*}
\Omega:=\{p\in \inte M \mid \rho_{\bm_{1}}(p)+\rho_{\bm_{2}}(p)=D\}.
\end{equation*}
Lemma \ref{lem:disconnected lemma} implies that
$\Omega$ is a non-empty closed subset of $\inte M$.

We show that
$\Omega$ is open in $\inte M$.
Take $p\in \Omega$.
For each $i=1,2$,
there exists a foot point $x_{p,i}\in \bm_{i}$ on $\bm_{i}$ of $p$ such that
$d_{M}(p,x_{p,i})=\rho_{\bm_{i}}(p)$.
From the triangle inequality,
we derive $d_{M}(x_{p,1},x_{p,2})=D$.
The normal minimal geodesic $\gamma:[0,D]\to M$ from $x_{p,1}$ to $x_{p,2}$ is orthogonal to $\bm$ at $x_{p,1}$ and at $x_{p,2}$.
Furthermore,
$\gamma|_{(0,D)}$ lies in $\inte M$ and passes through $p$.
There exists an open neighborhood $U$ of $p$ such that 
$U$ is contained in $\inte M$
and $\rho_{\bm_{i}}$ is smooth on $U$.
By using Lemma \ref{lem:Basic comparison},
we see $\Delta_{f}\, \rho_{\bm_{i}} \geq 0$ on $U$;
in particular,
$-(\rho_{\bm_{1}}+\rho_{\bm_{2}})$ is $f$-subharmonic on $U$.
Lemma \ref{lem:maximal principle} implies that 
$\Omega$ is open in $\inte M$.

Since $\inte M$ is connected,
we have $\inte M=\Omega$.
For each $x\in \bm_{1}$,
choose an orthonormal basis $\{e_{x,i}\}_{i=1}^{n-1}$ of $T_{x}\bm$.
Let $\{Y_{x,i}\}^{n-1}_{i=1}$ be the $\bm$-Jacobi fields along $\gamma_{x}$
with initial conditions $Y_{x,i}(0)=e_{x,i}$ and $Y_{x,i}'(0)=-A_{u_{x}}e_{x,i}$.
Using Lemma \ref{lem:Equality in Basic comparison},
for all $i$
we see $Y_{x,i}=F_{x}\,E_{x,i}$ on $[0,D]$,
where $\{E_{x,i}\}^{n-1}_{i=1}$ are the parallel vector fields along $\gamma_{x}$ with initial condition $E_{x,i}(0)=e_{x,i}$.
Moreover,
if $N\in (-\infty,1)$,
then for every $x\in \bm_{1}$
the function $f\circ \gamma_{x}$ is constant on $[0,D]$.
Define a map $\Phi:[0,D]\times \bm_{1}\to M$ by $\Phi(t,x):=\gamma_{x}(t)$.
The map $\Phi$ is a Riemannian isometry with boundary from $[0,D]\times_{F} \bm_{1}$ to $M$.
\end{proof}
\begin{rem}
Wylie \cite{W} has proved the same result as Theorem \ref{thm:disconnected splitting} when $\kappa=0$ and $\lambda=0$ (see Theorem 5.1 in \cite{W}).
\end{rem}

%% file: section6.tex
\section{Eigenvalue rigidity}\label{sec:Eigenvalue rigidity}
Let $M$ be an $n$-dimensional,
connected complete Riemannian manifold with boundary with Riemannian metric $g$,
and let $f:M\to \mathbb{R}$ be a smooth function.
\subsection{Lower bounds}
In \cite{Sa2},
the author has shown the following Picone type inequality proved by Allegretto and Huang \cite{AH} in the Euclidean case (see Theorem 1.1 in \cite{AH}, and Lemma 7.1 in \cite{Sa2}).
\begin{lem}[\cite{Sa2}]\label{lem:Picone identity}
Let $\varphi$ and $\psi$ be functions on $M$ that are smooth on a domain $U$ in $M$,
and satisfy $\varphi>0$ and $\psi \geq 0$ on $U$.
Then for all $p\in (1,\infty)$ we have the following inequality on $U$:
\begin{equation}\label{eq:Picone identity}
\Vert \nabla \psi \Vert^{p}\geq \Vert \nabla \varphi \Vert^{p-2} g\left(\nabla \left(  \psi^{p}\,  \varphi^{1-p} \right),\nabla \varphi \right).
\end{equation}
Moreover,
if the equality in $(\ref{eq:Picone identity})$ holds on $U$,
then for some constant $c\neq 0$
we have $\psi=c\, \varphi$ on $U$.
\end{lem}
We now give a proof of the inequality $(\ref{eq:eigenvalue rigidity})$ in Theorem \ref{thm:eigenvalue rigidity}.
\begin{prop}\label{prop:inequality in eigenvalue rigidity}
Suppose that
$M$ is compact.
Let $p\in (1,\infty)$.
Let $\kappa \in \mathbb{R}$ and $\lambda \in \mathbb{R}$ satisfy the subharmonic-condition.
For $N \in (-\infty,1]$
we suppose $\ric^{N}_{f,M}\geq \kappa$,
and $H_{f,\bm} \geq \lambda$.
For $D\in (0,\infty)$
we assume $\dm \leq D$.
Then we have $(\ref{eq:eigenvalue rigidity})$.
\end{prop}
\begin{proof}
Let $\varphi_{p,D}:[0,D]\to \mathbb{R}$ be a function satisfying (\ref{eq:model space eigenvalue problem}) for $\mu=\mu_{p,D}$.
We may assume $\varphi_{p,D}|_{(0,D]}>0$.
Then we see
$\varphi'_{p,D}|_{[0,D)}>0$.
Put $\Phi:=\varphi_{p,D}\circ \rho_{\bm}$.
Take a non-negative,
non-zero smooth function $\psi$ on $M$ whose support is compact and contained in $\inte M$.
By Lemma \ref{lem:Picone identity}
\begin{equation}\label{eq:eigenvalue rigidity4}
\Vert \nabla \psi \Vert^{p}\geq \Vert \nabla \Phi \Vert^{p-2} g\left(\nabla \left(  \psi^{p}\,  \Phi^{1-p} \right),\nabla \Phi \right)
\end{equation}
on $\inte M \setminus \cut \bm$.
From (\ref{eq:eigenvalue rigidity4}) and Proposition \ref{prop:global p-Laplacian comparison},
we derive
\begin{multline*}
\int_{M}\, \Vert \nabla \psi \Vert^{p}\, d\,m_{f}\geq \int_{M}\,\Vert \nabla \Phi \Vert^{p-2} g\left(\nabla \left(  \psi^{p}\,  \Phi^{1-p} \right),\nabla \Phi \right) \,d\,m_{f}\\
                                                                      \geq \int_{M}\,  \left( \psi^{p}\,  \Phi^{1-p} \right)  \left(-\left( \left(\varphi'_{p,D} \right)^{p-1} \right)' \circ \rho_{\bm}\right)   \,d\,m_{f}=\mu_{p,D}\,\int_{M}\, \psi^{p} \, d\,m_{f}.
\end{multline*}
It follows that $R_{f,p}(\psi)\geq \mu_{p,D}$.
Hence,
we arrive at (\ref{eq:eigenvalue rigidity}).
\end{proof}
\begin{rem}\label{rem:the equality case in eigenvalue rigidity}
In Proposition \ref{prop:inequality in eigenvalue rigidity},
we assume that
there exists a non-negative,
non-zero smooth function $\psi:M\to \mathbb{R}$ whose support is compact and contained in $\inte M$ such that $R_{f,p}(\psi)=\mu_{p,D}$.
Then the equality in (\ref{eq:eigenvalue rigidity4}) holds on $\inte M \setminus \cut \bm$,
and hence
for some constant $c\neq 0$
we have $\psi=c\, \Phi$ on $M$ (see Lemma \ref{lem:Picone identity}).
The equality case in (\ref{eq:global p-Laplacian comparison}) also happens (see Remark \ref{rem:Equality case in global p-Laplacian comparison}).
\end{rem}
\subsection{Equality cases}
For a positive number $D\in (0,\infty)$,
we put $S_{D}(\bm):=B_{D}(\bm)\setminus U_{D}(\bm)$.

For the proof of Theorem \ref{thm:eigenvalue rigidity},
we show the following lemma
concerning the $F$-model space introduced in Subsection \ref{subsec:Eigenvalue rigidity}:
\begin{lem}\label{lem:conclude rigidity}
Suppose that
$M$ is compact.
let $\kappa \in \mathbb{R}$ and $\lambda \in \mathbb{R}$ satisfy the subharmonic-condition.
For $N \in (-\infty,1]$
we suppose $\ric^{N}_{f,M}\geq \kappa$,
and $H_{f,\bm} \geq \lambda$.
Assume that
for some $D\in (0,\infty)$
we have $\cut \bm=S_{D}(\bm)$.
For each $x\in \bm$,
choose an orthonormal basis $\{e_{x,i}\}_{i=1}^{n-1}$ of $T_{x}\bm$,
and let $\{Y_{x,i}\}^{n-1}_{i=1}$ be the $\bm$-Jacobi fields along $\gamma_{x}$
with initial conditions $Y_{x,i}(0)=e_{x,i}$ and $Y_{x,i}'(0)=-A_{u_{x}}e_{x,i}$.
Assume further that
for all $i$
we have $Y_{x,i}=F_{x}\, E_{x,i}$ on $[0,D]$,
where $\{E_{x,i}\}^{n-1}_{i=1}$ are the parallel vector fields along $\gamma_{x}$ with initial condition $E_{x,i}(0)=e_{x,i}$.
Then $(M,d_{M})$ is an $F$-model space.
\end{lem}
\begin{proof}
First,
we assume that
$\bm$ is disconnected.
Let $\{\bm_{i}\}_{i=1,2,\dots}$ be the connected components of $\bm$.
Put $D_{1}:=\inf_{i=2,3,\dots}\, d_{M}(\bm_{1},\bm_{i})$.
By Theorem \ref{thm:disconnected splitting},
there exists a connected component $\bm_{1}$ such that
$(M,d_{M})$ is isometric to $([0,D_{1}] \times_{F} \bm_{1},d_{[0,D_{1}] \times_{F} \bm_{1}})$.
From $\cut \bm=S_{D}(\bm)$,
it follows that
$\dm=D$ and $D_{1}=2D$,
and hence $(M,d_{M})$ is an $F$-model space.

Next,
we assume that
$\bm$ is connected.
By $\cut \bm=S_{D}(\bm)$,
we have $\dm=D$.
From the property of Jacobi fields,
$S_{D}(\bm)$ is a smooth hypersurface in $M$,
and every point in $S_{D}(\bm)$ has two distinct foot points on $\bm$.
For every $x\in \bm$,
the vector $\gamma'_{x}(D)$ is orthogonal to $S_{D}(\bm)$;
and hence
the number of foot points on $\bm$ of $\gamma_{x}(D)$ is equal to two.
Now,
we define an involutive isometry $\sigma:\bm\to \bm$ without fixed points by $\sigma(x):=y$,
where $y$ is the foot point on $\bm$ of $\gamma_{x}(D)$
that is different from $x$.
We also define a map $\Phi:[0,2D]\times \bm \to M$ as follows:
If $t \in [0,D)$,
then $\Phi(t,x):=\gamma_{x}(t)$;
if $t \in (D,2D]$,
then $\Phi(t,x):=\gamma_{\sigma(x)}(2D-t)$.
We see that
$\Phi$ is surjective and continuous.
For all $x \in \bm$ and $t \in [0,2D]$
we have $\Phi(t,x)=\Phi(2D-t,\sigma(x))$.
Since for all $x\in \bm$ and $i$
we have $Y_{x,i}=F_{x}\, E_{x,i}$ on $[0,D]$,
the map $\Phi|_{[0,D)}$ gives an isometry between $(U_{D}(\bm),g)$ and the twisted product space $[0,D)\times_{F} \bm$.
Therefore,
$M$ is isometric to the quotient space $([0,2D]\times_{F} \bm)/G_{\sigma}$,
where $G_{\sigma}$ is the isometry group on $[0,2D]\times_{F} \bm$ of the identity and the involute isometry $\hat{\sigma}$ defined by $\hat{\sigma}(t,x):=(2D-t,\sigma(x))$.
This implies that
$(M,d_{M})$ is an $F$-model space.
We complete the proof.
\end{proof}
Furthermore,
we recall the following fact concerning eigenfunctions for the $(f,p)$-Laplacian (see e.g., \cite{Sa2}, \cite{T}):
\begin{prop}\label{prop:eigenfunction}
Suppose that
$M$ is compact.
Let $p\in (1,\infty)$.
Then there exists a non-negative,
non-zero function $\Psi$ in $W^{1,p}_{0}(M,m_{f})$ such that $R_{f,p}(\Psi)=\mu_{f,1,p}(M)$.
Moreover,
for some $\alpha \in (0,1)$
the function $\Psi$ is $C^{1,\alpha}$-H\"older continuous on $M$.
\end{prop}
Now,
we prove Theorem \ref{thm:eigenvalue rigidity} and Corollary \ref{cor:warped eigenvalue rigidity}.
\begin{proof}[Proof of Theorem \ref{thm:eigenvalue rigidity}]
Suppose that
$M$ is compact.
Let $p\in (1,\infty)$,
and let $\kappa \in \mathbb{R}$ and $\lambda \in \mathbb{R}$ satisfy the subharmonic-condition.
For $N \in (-\infty,1]$
we suppose $\ric^{N}_{f,M}\geq \kappa$,
and $H_{f,\bm} \geq \lambda$.
For $D\in (0,\infty)$,
we assume $\dm \leq D$.
From Proposition \ref{prop:inequality in eigenvalue rigidity},
we derive (\ref{eq:eigenvalue rigidity}).

We assume that
the equality in (\ref{eq:eigenvalue rigidity}) holds.
Proposition \ref{prop:eigenfunction} implies that
there exists a non-negative,
non-zero function $\Psi$ in $W^{1,p}_{0}(M,m_{f})$ such that
$R_{f,p}(\Psi)=\mu_{p,D}$ and $\Psi$ is $C^{1,\alpha}$-H\"older continuous on $M$.
Let $\varphi_{p,D}:[0,D]\to \mathbb{R}$ be a function satisfying (\ref{eq:model space eigenvalue problem}) for $\mu=\mu_{p,D}$,
and let $\varphi_{p,D}|_{(0,D]}>0$.
Putting $\Phi:=\varphi_{p,D}\circ \rho_{\bm}$,
we see that
$\Phi$ coincides with a constant multiplication of $\Psi$ on $M$ (see Remark \ref{rem:the equality case in eigenvalue rigidity});
in particular,
$\Phi$ is also $C^{1,\alpha}$-H\"older continuous on $M$.

For each $x\in \bm$,
we choose an orthonormal basis $\{e_{x,i}\}_{i=1}^{n-1}$ of $T_{x}\bm$.
Let $\{Y_{x,i}\}^{n-1}_{i=1}$ be the $\bm$-Jacobi fields along $\gamma_{x}$
with initial conditions $Y_{x,i}(0)=e_{x,i}$ and $Y_{x,i}'(0)=-A_{u_{x}}e_{x,i}$.
Then for all $i$
we have $Y_{x,i}=F_{x}\, E_{x,i}$ on $[0,\tau(x)]$,
where $\{E_{x,i}\}^{n-1}_{i=1}$ are the parallel vector fields along $\gamma_{x}$ with initial condition $E_{x,i}(0)=e_{x,i}$;
moreover,
if $N \in (-\infty,1)$,
then $f \circ \gamma_{x}$ is constant on $[0,\tau(x)]$ (see Remark \ref{rem:the equality case in eigenvalue rigidity}).

We show $\cut \bm=S_{D}(\bm)$.
From $\dm \leq D$
we deduce $S_{D}(\bm)\subset \cut \bm$.
We prove the opposite.
Take $p_{0}\in \cut \bm$.
By the property of Jacobi fields,
$\rho_{\bm}$ is not differentiable at $p_{0}$.
From the regularity of $\Phi$,
it follows that $\varphi'_{p,D}(\rho_{\bm}(p_{0}))=0$;
in particular,
$\rho_{\bm}(p_{0})=D$.
Hence,
$\cut \bm=S_{D}(\bm)$.
This implies
$\dm=D$.
By Lemma \ref{lem:conclude rigidity},
we complete the proof of Theorem \ref{thm:eigenvalue rigidity}.
\end{proof}
\medskip

\begin{proof}[Proof of Corollary \ref{cor:warped eigenvalue rigidity}]
By Theorem \ref{thm:eigenvalue rigidity} and Proposition \ref{prop:twisted to warped},
we complete the proof of Corollary \ref{cor:warped eigenvalue rigidity}.
\end{proof}
By Theorem \ref{thm:eigenvalue rigidity} and $\mu_{2,D}= \pi^{2}(2D)^{-2}$,
we have the following:
\begin{cor}\label{cor:Li-Yau}
Let $M$ be a connected complete Riemannian manifold with boundary,
and let $f:M\to \mathbb{R}$ be a smooth function.
Suppose that
$M$ is compact.
Let $\kappa \in \mathbb{R}$ and $\lambda \in \mathbb{R}$ satisfy the subharmonic-condition.
For $N\in (-\infty,1]$
we suppose $\ric^{N}_{f,M}\geq \kappa$,
and suppose $H_{f,\bm} \geq \lambda$.
For $D\in (0,\infty)$
we assume $\dm \leq D$.
Then
\begin{equation}\label{eq:Li-Yau estimate}
\mu_{f,1,2}(M)\geq \frac{{\pi}^{2}}{4D^{2}}.
\end{equation}
If the equality in $(\ref{eq:Li-Yau estimate})$ holds,
then $\dm=D$,
and the metic space $(M,d_{M})$ is an $F$-model space.
Moreover,
if $N\in (-\infty,1)$,
then for every $x\in \bm$
the function $f\circ \gamma_{x}$ is constant on $[0,D]$;
in particular,
$(M,d_{M})$ is an equational model space.
\end{cor}

\subsection{Explicit lower bounds}\label{sec:Eigenvalue estimates}
Let $\Omega$ be a relatively compact domain in $M$ such that
$\partial \Omega$ is a smooth hypersurface in $M$ satisfying $\partial \Omega \cap \bm=\emptyset$.
For the canonical Riemannian volume measure $\vol_{\partial \Omega}$ on $\partial \Omega$,
let $m_{f,\partial \Omega}:=e^{-f|_{\partial \Omega}}\, \vol_{\partial \Omega}$.
Put
\begin{equation}\label{eq:diameter of Omega}
\delta_{1}(\Omega):=\inf_{p\in \Omega}\, \rho_{\bm}(p),\quad \delta_{2}(\Omega):=\sup_{p\in \Omega} \,\rho_{\bm}(p).
\end{equation}

In the usual weighted case,
the author \cite{Sa2} has proved the following volume estimate (see Propositions 8.1 and 8.2 in \cite{Sa2}):
\begin{prop}[\cite{Sa2}]\label{prop:usual Kasue volume estimate}
Let $M$ be an $n$-dimensional,
connected complete Riemannian manifold with boundary,
and let $f:M\to \mathbb{R}$ be a smooth function.
For $N \in [n,\infty]$
we suppose $\ric^{N}_{f,M}\geq 0$,
and $H_{f,\bm} \geq 0$.
Let $\Omega$ be a relatively compact domain in $M$ such that
$\partial \Omega$ is a smooth hypersurface in $M$ satisfying $\partial \Omega \cap \bm=\emptyset$.
Then
\begin{equation*}\label{eq:usual Kasue volume estimate}
m_{f}(\Omega) \leq m_{f,\partial \Omega}\, (\partial \Omega) \,\left(\delta_{2}(\Omega)-\delta_{1}(\Omega)   \right),
\end{equation*}
where $\delta_{1}(\Omega)$ and $\delta_{2}(\Omega)$ are the values defined as $(\ref{eq:diameter of Omega})$.
\end{prop}
Kasue \cite{K3} has obtained Proposition \ref{prop:usual Kasue volume estimate} in the unweighted case.

In our setting,
we have the following volume estimate:
\begin{prop}\label{prop:Kasue volume estimate}
Let $M$ be a connected complete Riemannian manifold with boundary,
and let $f:M\to \mathbb{R}$ be a smooth function.
Let $\kappa \in \mathbb{R}$ and $\lambda \in \mathbb{R}$ satisfy the subharmonic-condition.
For $N \in (-\infty,1]$
we suppose $\ric^{N}_{f,M}\geq \kappa$,
and $H_{f,\bm} \geq \lambda$.
Let $\Omega$ be a relatively compact domain in $M$ such that
$\partial \Omega$ is a smooth hypersurface in $M$ satisfying $\partial \Omega \cap \bm=\emptyset$.
Then
\begin{equation*}\label{eq:Kasue volume estimate}
m_{f}(\Omega) \leq m_{f,\partial \Omega}\, (\partial \Omega) \,\left(\delta_{2}(\Omega)-\delta_{1}(\Omega)   \right),
\end{equation*}
where $\delta_{1}(\Omega)$ and $\delta_{2}(\Omega)$ are the values defined as $(\ref{eq:diameter of Omega})$.
\end{prop}
We can prove Proposition \ref{prop:Kasue volume estimate}
only by replacing the role of the comparison result in the usual weighted case
with that of Lemma \ref{lem:Laplacian comparison} in the proof of Proposition \ref{prop:usual Kasue volume estimate}.
We omit the proof.

From Proposition \ref{prop:usual Kasue volume estimate},
the author \cite{Sa2} has derived the following estimate in the usual weighted case (see Theorems 8.4 and 8.5 in \cite{Sa2}):
\begin{thm}[\cite{Sa2}]\label{thm:usual p-Laplacian1}
Let $M$ be an $n$-dimensional,
connected complete Riemannian manifold with boundary,
and let $f:M\to \mathbb{R}$ be a smooth function.
Suppose that
$\bm$ is compact.
Let $p\in (1,\infty)$.
For $N \in [n,\infty]$
we suppose $\ric^{N}_{f,M}\geq 0$,
and $H_{f,\bm} \geq 0$.
For $D\in(0,\infty)$
we assume $D(M,\bm)\leq D$.
Then we have $\mu_{f,1,p}(M)\geq (pD)^{-p}$.
\end{thm}
The author \cite{Sa1} has shown Theorem \ref{thm:usual p-Laplacian1} in the unweighted case.

In our setting,
we can prove the following result
by using Proposition \ref{prop:Kasue volume estimate} instead of Proposition \ref{prop:usual Kasue volume estimate}
in the proof of Theorem \ref{thm:usual p-Laplacian1}.
The argument is in \cite{Sa2}.
\begin{thm}\label{thm:p-Laplacian1}
Let $M$ be a connected complete Riemannian manifold with boundary,
and let $f:M\to \mathbb{R}$ be a smooth function.
Suppose that
$\bm$ is compact.
Let $p\in (1,\infty)$,
and let $\kappa \in \mathbb{R}$ and $\lambda \in \mathbb{R}$ satisfy the subharmonic-condition.
For $N \in (-\infty,1]$
we suppose $\ric^{N}_{f,M}\geq \kappa$,
and $H_{f,\bm} \geq \lambda$.
For $D\in(0,\infty)$
we assume $D(M,\bm)\leq D$.
Then we have $\mu_{f,1,p}(M)\geq (pD)^{-p}$.
\end{thm}